\documentclass{article}
\usepackage{amsfonts}
\usepackage{amsmath}
\usepackage{amssymb}
\usepackage[utf8]{inputenc}
\usepackage{amsthm}
\usepackage{xcolor}
\usepackage{soul}
\usepackage{hyperref}
\usepackage{geometry}
\usepackage{xcolor}

\newtheorem{theorem}{Theorem}

\newtheorem{lemma}[theorem]{Lemma}
\newtheorem{proposition}[theorem]{Proposition}

\newtheorem{assumption}{Assumption}
\theoremstyle{remark}
\newtheorem{remark}{Remark}
\theoremstyle{definition}

\begin{document} 

\title{Averaging Principle for Ordinary Differential Equations in a Multiscale Random Environment}

\author{Vincent Kagan$^{1}$}

\footnotetext[1]{Universit\'e de Lorraine, CNRS, Inria, IECL, F-54000 Nancy, France \\
  E-mail: vincent.kagan@univ-lorraine.fr}
  
\maketitle

\begin{abstract}
We study slow–fast stochastic systems in which a fast process, evolving within several ergodic components, drives the transition rates of a slow jump process switching between these components. Considering a differential equation whose vector field depends on both processes, we show that, under suitable ergodicity and regularity conditions, its solutions converge in law in $D([0,T],\mathbb R^d)$ to an averaged system depending only on the slow process. The averaged drift is obtained by averaging the original vector field over the invariant measure of the fast process within each component. Our results extend classical averaging theory to systems where the vector field is explicitly influenced by a fast random environment with multi-component dynamics.
\end{abstract}

\section{Introduction}

In this paper we consider a stochastic process $X$ taking values in the disjoint union
\[
E:=\bigsqcup_{i\in I} E_i,
\]
such that for all $i\in I$ and $x\in E_i$,
\begin{align*}
\mathbf P_x\big(X_t\in E_i,\forall t\geq0\big)=1.
\end{align*}
As in~\cite{kagan2025averaging}, we construct a jump Markov process $\mathbf X^1$ taking values in $E$ by using the piecing-out procedure introduced by Ikeda, Nagasawa and Watanabe~\cite{MB2}. This construction is closely related to the general theory of piecewise deterministic Markov processes (PDMPs) developed by Davis~\cite{Davis1984}, where deterministic flows are interlaced with random jump mechanisms. The trajectory of $\mathbf X^1$ between two successive jumps is governed by the process $X$. We then consider its accelerated version $\mathbf X^n$ where we replace $X$ by $X^n$ defined by
\begin{align*}
X_t^n=X_{nt},\qquad t\geq0.
\end{align*}

We introduce the associated index process $\mathbf I^n$ with values in $I$ by
\begin{align*}
\mathbf I_t^n=i
\quad\Longleftrightarrow\quad
\mathbf X_t^n\in E_i.
\end{align*}
In this way, $(\mathbf X_t^n,\mathbf I_t^n)_{t\geq0}$ is a two-time-scale stochastic system, where $\mathbf I^n$ is a slow jump process on the index set $I$, whose transition rates depend on the position of the fast process $\mathbf X^n$. Denoting by $(\tau_k^n)_{k\geq0}$ the jump times of $\mathbf I^n$, with $\tau_0^n=0$, the process $\mathbf X^n$ evolves between two jumps according to a dynamics which is assumed to be ergodic in later assumptions, on the component $E_{\mathbf I_t^n}$.

\medskip

We consider the following random ordinary differential equation on $\mathbb R^d$:
\begin{align*}
\frac{\mathrm d\mathbf Z_t^{n,z}}{\mathrm dt}
=
F_{\mathbf I_t^n}(\mathbf Z_t^{n,z},\mathbf X_t^n),
\qquad
\mathbf Z_0^{n,z}=z\in\mathbb R^d.
\end{align*}

Averaging principles for multiscale dynamical systems originate in the classical theory of nonlinear oscillations developed by Bogoliubov and Mitropolsky~\cite{BogoliubovMitropolsky1961}, and were later formalized in deterministic averaging theory by Sanders, Verhulst and Murdock~\cite{sanders2007averaging}. These methods now belong to the broader framework of random dynamical systems, perturbation theory and homogenization theory, see, for instance, Arnold~\cite{Arnold1998}, Khasminskii~\cite{khasminskii2011stochastic} and Pavliotis and Stuart~\cite{PavliotisStuart2008}.

In stochastic settings, the first convergence results for systems of this type are due to Khasminskii~\cite{Hasminskii1966a}, who proved convergence in mean for differential equations with rapidly oscillating random coefficients. In~\cite{Hasminskii1966b}, he further established convergence rates and central limit theorems for such equations. These works constitute one of the foundations of stochastic averaging theory. Another major contribution is the martingale problem approach introduced by Kurtz~\cite{Kurtz2006}, which provides a general weak convergence framework for averaging procedures. This perspective was subsequently developed in the convergence theory of Markov processes by Ethier and Kurtz~\cite{EthierKurtz2005}. A fundamental extension of averaging methods to random perturbations of dynamical systems is given by Freidlin and Wentzell~\cite{bookFW}, where convergence in probability of
\[
\sup_{t\in[0,T]}
|\mathbf Z_t^n-\bar{\mathbf Z}_t|
\]
is established under suitable assumptions. We also refer to Freidlin~\cite{Freidlin1985} for related developments connecting averaging methods and partial differential equations.

\medskip

Under suitable assumptions, we prove that the process $\mathbf Z^{n,z}$ converges in law to an averaged process $\bar{\mathbf Z}^z$ solving
\begin{align*}
\frac{\mathrm d\bar{\mathbf Z}_t^z}{\mathrm dt}
=
\bar F_{\mathbf I_t}(\bar{\mathbf Z}_t^z),
\qquad
\bar{\mathbf Z}_0^z=z\in\mathbb R^d,
\end{align*}
where $(\mathbf I_t)_{t\geq0}$ denotes the limiting regime process with jump times $(\tau_k)_{k\geq0}$, and where for all $i\in I$ and $z\in\mathbb R^d$,
\begin{align*}
\bar F_i(z)
=
\int_{E_i}
F_i(z,x)\,\mu_i(\mathrm dx).
\end{align*}
More precisely, for fixed $k\ge1$ and $t_1<\cdots<t_N\leq\tau_k^n$,
\[
(\mathbf Z_{t_1}^n, \dots, \mathbf Z_{t_N}^n)
\Rightarrow
(\bar{\mathbf Z}_{t_1}, \dots, \bar{\mathbf Z}_{t_N}).
\]
We further show that this convergence holds in the Skorokhod space $D([0,T];\mathbb R^d)$ endowed with the $J_1$ topology, namely
\[
\mathbf Z^{n,z}
\Rightarrow
\bar{\mathbf Z}^z
\qquad
\text{in }
D([0,T];\mathbb R^d),
\quad
T>0.
\]

\medskip

A large literature has developed around averaging principles for multiscale stochastic systems and Markov-modulated dynamics. Early results on fast-switching dynamical systems were obtained by Sarafyan and Skorokhod~\cite{Skorokhod1988}. Averaging principles for PDMPs and jump-driven systems have been studied in several complementary directions.

The works of Faggionato, Gabrielli and Ribezzi Crivellari~\cite{faggionato2008averaging} on averaging and large deviation principles for fully-coupled piecewise deterministic Markov processes (with applications to molecular motors), and of Pakdaman, Thieullen and Wainrib~\cite{Pakdaman2012} on asymptotic expansion and central limit theorem for multiscale piecewise-deterministic Markov processes, provide closely related frameworks. Their settings are particularly relevant as applications of averaging results for PDMPs, and our work can be seen as extending this line of research to random ordinary differential equations driven by a jump process with fast ergodic dynamics on multiple components.

Related frameworks for diffusion processes evolving in partitioned domains were developed by Bobrowski, Kaźmierczak and Kunze~\cite{Bobrowski2017}. Averaging for slow--fast piecewise deterministic Markov processes with an attractive boundary was recently investigated by Genadot~\cite{Genadot}. More recently, Mao and Shao~\cite{mao2024} established averaging principles for two-time-scale regime-switching diffusions, proving both $L^1$ and weak convergence under suitable ergodicity assumptions. Bao, Yin and Yuan in \cite{bao2017two} studied diffusions whose slow component is driven by an $\alpha$-stable Lévy process, establishing $L^p$ convergence in the case where the fast process is a Markov process taking values in a finite state space and is independent of the slow component, with either a single irreducible class or multiple irreducible classes, as in the present paper. They also considered the case where the dynamics of the slow component are governed by a diffusion driven by $\alpha$-stable noise, with the noise depending on the slow component. Their analysis relies, among other techniques, on a semigroup approach and on the use of a mild Itô formula.

Cerrai has also made substantial contributions to the development of averaging principles for slow–fast diffusion processes driven by Gaussian noise, particularly for reaction–diffusion equations, notably in the series of papers \cite{cerrai2009khasminskii,cerrai2009averaging}. These works establish convergence in distribution in the space of continuous functions, as well as convergence in probability, in the case where the slow component does not depend on the fast component. In these works, and in one of the first contributions in the literature to do so, Cerrai considers a framework in which the processes evolve in an infinite-dimensional space.

To establish these averaging principles, Cerrai follows a classical approach consisting of freezing the slow component and studying the existence of an invariant measure for the associated fast process, together with the existence and suitable properties of the averaged solutions that allow the convergence to be established. Once again, Cerrai adopts a semigroup approach and works with mild solutions. Subsequently, Cerrai studied the normal deviations associated with this type of stochastic differential equation system in \cite{cerrai2009normal}.

Monmarché and Strickler~\cite{Monmarche2025} investigated averaging and invariant measure expansions for Markov-modulated ODEs at high frequency. We also refer to the monograph of Mao and Yuan~\cite{bookMao} for a comprehensive treatment of stochastic differential equations with Markovian switching.

Such principles have also been developed in recent years beyond the classical Markovian framework. In the context of stochastic differential equations, averaging has been established for equations with locally dependent or locally Lipschitz coefficients~\cite{LIU20202910}. Results for two-time-scale stochastic functional differential equations with past-dependent switching have been obtained in~\cite{Wu2025functional}. Averaging for multiscale stochastic PDEs under local monotonicity conditions has been studied in~\cite{DCDS2025SPDE}.

Recent work has extended averaging principles to stochastic partial differential equations with non-Lipschitz jump noise see ~\cite{cerrai2003stochastic}. Weak and strong averaging principles for 2D Boussinesq equations with non-Lipschitz Poisson jump noise have been established in~\cite{Shi2026Boussinesq}. Averaging principles for slow-fast fractional stochastic differential equations have been investigated in~\cite{Breher2025fractional}. Results for slow-fast systems of PDEs with rough drivers using interpolation spaces have been obtained in~\cite{Li2026roughPDE}.

Averaging principles have also been extended to time-inhomogeneous multi-scale SDEs with partially dissipative coefficients~\cite{Sun2025timeinhomogeneous}.

Closser to the present setting, averaging principles for jump processes whose transition mechanisms depend on a rapidly mixing environment have recently been investigated in~\cite{kagan2025averaging}. The present work may be viewed as a further step in this direction, combining fast ergodic dynamics, regime-switching mechanisms and random ordinary differential equations driven by a jump process evolving on multiple time scales.

While classical averaging results focus on convergence to an averaged limit, recent work has also studied fluctuations around the averaged limit. Fluctuations from a random fractional averaging limit have been analyzed in~\cite{Li2025fluctuations}, providing central limit theorem-type results for averaging with fractional noise.

\medskip

The theoretical framework developed in this paper is motivated by several concrete application domains. In neuroscience, slow--fast stochastic systems are used to model excitable neurons and spike generation, with averaging providing effective descriptions of neuronal dynamics~\cite{Desroches2016neuroscience}. In epidemiology, stochastic differential equations with Markovian switching have been applied to SIS epidemic models with telegraph noise~\cite{AIMS2020SIS}. 

In systems biology, multiscale stochastic models describe biochemical reaction networks and gene expression at the single-cell level~\cite{Wilkinson2018biology}. Stochastic models of resource allocation in chemical reaction networks have been studied using multiscale approaches~\cite{Fromion2025resource}. Scaling methods and asymptotic analysis of multiscale approximations to reaction networks provide rigorous convergence results for biochemical systems with multiple time scales~\cite{Laurence2026scaling,Ball2006multiscale}.

In population dynamics, averaging principles for stochastic slow-fast models have been applied to study the development of ovarian follicles~\cite{Ballif2022population}.

Large deviation theory for slow--fast McKean--Vlasov equations also has applications in statistical mechanics~\cite{Swan2023LDP}.

These application domains illustrate the broad relevance of averaging principles for multiscale stochastic systems. They also motivate the development of general frameworks that can accommodate various types of fast ergodic dynamics and regime-switching mechanisms.

\section{Description of the model}

This section provides a detailed description of the model and a brief outline of its construction; see \cite{kagan2025averaging} for further details.

\medskip

\textbf{Description of the coupled process $(\mathbf X^n,\mathbf I^n)$.} As mentioned in the introduction, we construct the process $\mathbf X^1$ on $\sqcup_{i\in I} E_i$, where $(E_i)_{i\in I}$ is a family of measurable sets indexed by a measurable index set $I$, by piecing out a stochastic process $X$. 

\medskip

Let $(E_i,\mathcal{E}_i)_{i\in I}$ a collection of measurable spaces indexed by a measurable set $(I,\mathcal I)$.  We define the measurable space $(E,\mathcal{E})$ with $E:=\sqcup_{i\in I}E_i$ (here $\sqcup$ denotes the disjoint union) endowed with the $\sigma$-algebra $\mathcal{E}$  defined by
\[
\mathcal{E}:=\{S\subset E:S\cap E_i\in\mathcal{E}_i\;\textup{for each}\;i\in I\}
\]
In particular, with this choice of $\sigma$-algebra, the application 
\[
\phi:x\in E\to i\in I\text{ such that }x\in E_i
\]
is measurable.

In the remainder of this section, we describe an $E\times I$-valued stochastic process $(\mathbf X^1, \mathbf I^1)$, where $\mathbf I^1$ is a jump process on $I$ whose transitions depend on the position of $\mathbf X^1$ and, when $\mathbf I^1$ is in $i\in I$, $\mathbf X^1$ evolves on $E_i$ according to some progressively measurable dynamic up to the next jump time of $\mathbf I^1$. Then $(\mathbf X^1, \mathbf I^1)$ transitions from $E_i \times \{i\}$ to a new position in $E \times I$.
 
 In order to describe our base stochastic process $(\mathbf{X}^1, \mathbf I^1)$, we need three ingredients: the dynamic of $\mathbf X^1$ on each $E_i$ between the jumps of $\mathbf I^1$, the law of the jump times, and the law of the  new position in $E \times I$ after the jump. Note that it is actually sufficient to describe the dynamic of $\mathbf X^1$ on $E$, since $\mathbf I^1=\phi(\mathbf X^1)$.
 
\paragraph{Dynamic of $\mathbf X^1$ between the jumps of $\mathbf I^1$} Let $(V,\mathcal F)$ be a measurable space endowed with a filtration $(\mathcal F_t)_{t\geq 0}$ and a family of probability measures $(\mathbf{P}_x)_{x\in E}$. We assume that we are given a progressively measurable  process 
\[
X=(V, \mathcal F,(\mathcal F_t)_{t\geq 0},(\mathbf{P}_x)_{x\in E},(X_t)_{t\geq0})
\]
with values in the measurable space $(E,\mathcal E)$ and such that, for all $i\in I$ and $x_i\in E_i$,
\begin{equation}
\label{eq:stable-E-i}
\mathbf{P}_{x_i}(\forall t\geq0,X_t\in E_i)=1.
\end{equation}

\paragraph{Law of the jump time of $\mathbf{I}^1$}
We consider an absolutely continuous cumulative distribution function $G:\mathbb{R}_+\longrightarrow [0,1]$, and a measurable rate function $b: E\rightarrow [0,\infty)$ such that $\int_{0+} b(X_t)\,\mathrm dt<\infty$ $\mathbf P_x$-almost surely, for all $x\in E$. For all $x \in E$, we let $\nu_x$ denote the distribution of the random variable 
\[
\tau_1^1 = \inf\{ t \geq 0 \: : G\Big(\int_0^t b (X_s) \, \mathrm ds\Big)\geq  U \}
\]
under $\mathbf P_x$, and where $U$ is an independent random variable with uniform law on $[0,1]$. The idea is that $\nu_x$ is the law of the time elapsed between the $k$-th and the $(k+1)$-th jump times whenever the process is at $x$ after the $k$-th jump time. Note also that in the particular case where $G(x) = 1 - e^{-x}$, we recover the classical setting of a process with exponential jump rate~$b$.

\paragraph{Transition kernel}  We also assume that we are given a transition kernel $\pi$ from $E$ to $E$: if the process is at some position $x \in E$ just before the transition time, then its new position is chosen according to $\pi(x, \cdot)$.

\medskip

In the following proposition, $\Delta\notin E$ plays the role of a cemetery point.
\begin{proposition}
There exists a random process $\mathbf{X}^1 = (\Omega,\mathcal{F}^1,(\mathbf{X}_t^1)_{t\geq0}, (\mathbf{P}_x^1)_{x\in E})$ with state space $E\cup \{\Delta\}$ with lifetime $T_\Delta^1:=\inf\{t\geq 0,\mathbf X_t^1=\Delta\}$, and a random sequence of times $(\tau_k^1)_{k \geq 0}$ such that
\begin{enumerate}
    \item $\tau_0^1 = 0$;
    \item for all $k \geq 0$, conditionally on $\mathbf{X}_{\tau_k^1}^1$
    the distribution function of $\tau_{k+1}^1 - \tau_k^1$ is given by $\nu_{\mathbf{X}_{\tau_k^1}^1}$;
    \item for all $k \geq 0$, conditionally on $\mathbf{X}_{\tau_k^1}^1$,  $(\mathbf{X}_{t-\tau_k^1}^1)_{t\in [\tau_k^1, \tau_{k+1}^1)}$ is distributed as $(X_t)_{t\in[0,\zeta)}$ under $\mathbf P_{\mathbf{X}_{\tau_k^1}^1}$;
    \item at time $t = \tau_{k+1}^1$, $\mathbf{X}_t^1$ transitions to a random position selected according to the kernel $\pi(\mathbf{X}_{\tau_{k+1}^1-}^1,\,\cdot\,)$;
    \item $T_\Delta^1=\lim_{n\to+\infty} \tau_k^1$.
\end{enumerate}
In addition, this process satisfies the strong Markov property at its transition times $(\tau_k^1)_{k \geq 0}$.
\end{proposition}

Then we consider the accelerated version $(\mathbf X^n,\mathbf I^n)$ of $(\mathbf X^1,\mathbf I^1)$ where we take $X^n=(X_{nt})_{t\geq0}$ instead of $X$.
We denote by $(\tau_k^n)_{k\geq0}$ the jump sequence of $\mathbf I^n$.

\medskip

\textbf{Definition of the stochastic differential equation $\mathbf Z^n$.}
From the process $(\mathbf X^n,\mathbf I^n)$ on $E \times I$, we define, for each initial condition $z \in\mathbb R^d$, a process $(\mathbf Z_t^{n,z})_{t \ge 0}$ with values in $Z$ as the solution to the differential equation
\begin{align*}
\frac{\mathrm d\mathbf Z_t^{n,z}}{\mathrm dt}
= F_{\mathbf I_t^n}(\mathbf Z_t^{n,z}, \mathbf X_t^n),
\quad
\mathbf Z_0^{n,z} = z,
\end{align*}
where $(F_i)_{i \in I}$ is a family of measurable functions from $Z \times E$ to $Z$.

\medskip

We assume that for all $i \in I$,
\begin{align}
\label{hyp:Lip}
L_i := \sup_{x \in E_i} \|F_i(\,\cdot\,,x)\|_{\mathrm{Lip}} < \infty.
\end{align}

\begin{remark}
Under assumption \eqref{hyp:Lip}, for every $n \in \mathbb N$ and every initial condition $z \in\mathbb R^d$, the equation above admits a unique adapted continuous solution $(\mathbf Z_t^{n,z})_{t \ge 0}$.
\end{remark}

The uniqueness of the adapted continuous solution $(\mathbf Z_t^{n,z})_{t \ge 0}$ for every initial condition $z \in\mathbb R^d$ implies that the process $(\mathbf Z^n,\mathbf X^n)$ satisfies the strong Markov property at the stopping time $\tau_1^n$.

\begin{proposition}
Let $(\mathbf X,\mathbf I)$ be a stochastic process with values in $E \times I$, and let $\tau_1$ be a stopping time such that $(\mathbf X,\mathbf I)$ satisfies the strong Markov property at time $\tau_1$. Let $\mathbf Z$ be defined as the solution to
\[
\frac{\mathrm d\mathbf Z_t}{\mathrm dt} = F_{\mathbf I_t}(\mathbf Z_t,\mathbf X_t), \quad \mathbf Z_0 = z,
\]
where $(F_i)_{i \in I}$ is a family of measurable functions from $Z \times E$ to $Z$, and assume that this equation admits a unique adapted solution. Then the process $(\mathbf Z,\mathbf X)$ satisfies the strong Markov property at time $\tau_1$.
\end{proposition}

\begin{proof}
By the uniqueness of the adapted solution, the process $\mathbf Z$ is a measurable functional of the trajectory $(\mathbf X_s,\mathbf I_s)_{0 \le s \le t}$ for each $t \ge 0$. In particular, $\mathbf Z_{\tau_1}$ is measurable with respect to $\mathcal F_{\tau_1}^{\mathbf X,\mathbf I}$. Moreover, for all $t \ge 0$, we have
\[
\mathbf Z_{\tau_1+t}
= \mathbf Z_{\tau_1} + \int_0^t F_{\mathbf I_{\tau_1+s}}(\mathbf Z_{\tau_1+s}, \mathbf X_{\tau_1+s})\, ds.
\]
Hence, the future evolution of $(\mathbf Z,\mathbf X)$ after time $\tau_1$ is entirely determined by $(\mathbf Z_{\tau_1}, \mathbf X_{\tau_1})$ together with the future trajectory $(\mathbf X_{\tau_1+s}, \mathbf I_{\tau_1+s})_{s \ge 0}$. By the strong Markov property of $(\mathbf X,\mathbf I)$ at time $\tau_1$, the conditional law of $(\mathbf X_{\tau_1+t}, \mathbf I_{\tau_1+t})_{t \ge 0}$ given $\mathcal F_{\tau_1}^{\mathbf X,\mathbf I}$ depends only on $(\mathbf X_{\tau_1}, \mathbf I_{\tau_1})$. Therefore, the conditional law of $(\mathbf Z_{\tau_1+t}, \mathbf X_{\tau_1+t})_{t \ge 0}$ given $\mathcal F_{\tau_1}^{\mathbf X,\mathbf I}$ depends only on $(\mathbf Z_{\tau_1}, \mathbf X_{\tau_1})$. This proves that $(\mathbf Z,\mathbf X)$ satisfies the strong Markov property at time $\tau_1$.
\end{proof}

\medskip

\textbf{Definition of the limiting index process.} 
Let $(\Omega,\mathcal F, (\mathbf I_t)_{t\geq 0},(\mathbf P_i)_{i\in I})$ be a semi-Markov process whose dynamics are defined as follows:  for all $i\in I$, under $\mathbf P_i$  $(\mathbf I_t)_{t\geq 0}$ is an $I$-valued pure jump semi-Markov process 
with $\mathbf P_i(\mathbf I_0=i)=1$, such that its  first jump time $\tau_1$ is defined by
\begin{align*}
\tau_1=\inf\big\{t>0:G\big(t\mu_i(b)\big)\geq U\big\}.
\end{align*}
and such that its jumps transitions probabilities $(P(i,j))_{i,j\in I}$ are given by
\begin{equation}
\label{eq:jumpMarkov}
P(i,J)=\int_E\frac{\mu_i(\mathrm dx)}{\mu_i(b)}\,b(x)\,\nu(x,J), \quad \forall J \in\mathcal{I}.
\end{equation}
We note $(\tau_k)_{k\geq1}$ the sequence of jump times of $\mathbf I$. We define the explosion time of $\mathbf I$ as
\[
    \tau_\infty=\lim_{k\to\infty}\tau_k\in[0,+\infty],
\]
and we set $\mathbf I_t=\Delta_{I}$ $\forall t\geq \tau_\infty$, where $\Delta_{I}\notin I$ is a cemetery point. 

\medskip 
\noindent \textbf{Definition of the limiting Markov chain on $E$.} 
Let $(Y_n)_{n \geq 0}$ the Markov chain on $E\cup \{\Delta\}$, whose transition kernel $Q$ is defined as follows: for all $x\in E\cup\{\Delta\}$ and bounded measurable function $f:E \cup \{\Delta\} \to\mathbb R$,
\[
Qf(x) = \begin{cases}
\int_E\frac{\mu_{\phi(x)}(\mathrm dy)}{\mu_{\phi(x)}(b)}\,b(y)\,\pi(y,f),\quad \text{ if }\mu_{\phi(x)}(b)>0, \\
f(\Delta),\quad \text{ if }\mu_{\phi(x)}(b)=0,
\end{cases}
\]
where $b(\Delta):=0$.
Informally, starting from $Y_0 =x$, with probability $\frac{\mu_{\phi(x)}(b \,\cdot)}{\mu_{\phi(x)}}$, one chooses a position, say $Y_1^-$, before the jump, and then picks $Y_1$ according to $\pi(Y_1^-,\,\cdot\,)$.

Note that, for all $x\in E$, $Q f(x)$ depends on $x$ only through $\phi(x)$, so that the law of $Y_{n+1}$ depends on $Y_n$ only through $\phi(Y_n)$.
%On peut décomposer l'action de $Q$ en deux morceaux:
We also observe that
\[
Pf(\phi(x)) = Q (f \circ \phi)(x)
\]
so that $\mathbf I_{\tau_n} = \phi(\mathbf{Y}_{n})$, for all $n\geq 0$.

\medskip

\textbf{Definition of the averaged limiting differential equation.} Finally we defined the average limiting process as the solution of the stochastic differential equation:
\begin{align*}
\frac{\mathrm d\bar{\mathbf Z}_t^z}{\mathrm dt}=\bar F_{\mathbf I_t}(\bar{\mathbf Z}_t^z),\quad\bar{\mathbf Z}_0^z=z
\end{align*}
where for all $i\in I,z\in\mathbb R^d$
\begin{align*}
\bar F_i(z)=\int_EF_i(z,x)\,\mu_i(\mathrm dx)
\end{align*}

\section{Main results}

In this section, we first introduce the assumptions required for our first convergence result, presented in the second subsection: the convergence of the finite-dimensional distributions of $\mathbf Z^n$. The corresponding proofs are given in Subsections \ref{subsec:proofConv} and \ref{subsec:proofConvFD}. Finally, in the last subsection, we prove the convergence in law of the sequence $(\mathbf Z^n)_{n\geq1}$ in the Skorokhod topology.

\subsection{Assumptions}

In the remaining of this paper, we assume the following ergodic property for the process restricted to each $E_i$. Let $\mathcal{A}_i$ the subset of functions $f \in L^1(\mu_i)$ such that, for all $x \in E_i$ and all $t \geq 0$,
\begin{equation}
    \label{eq:finite-integral-trajectory}
    \mathbf{P}_{x_i} \Big( \int_0^t | f(X_s) |\, \mathrm ds < + \infty \Big) = 1.
\end{equation}
\begin{assumption}
\label{assumption:ergo}

For all $i\in I$, there exists a probability distribution $\mu_i$ on $E_i$ such that, for all $h\in \mathcal{A}_i$,
 and $x_i\in E_i$
\begin{equation}
\label{hyp:ergo}
\lim_{t\to\infty}\frac{1}{t}\int_0^t h(X_s)\,\mathrm ds=\mu_i(h)\quad\mathbf{P}_{x_i}\textup{-a.s.}
\end{equation}
where $\mu_i(h):=\int_E h(x)\,\mu_i(\mathrm dx)$.
\end{assumption}
We also make the following assumption on the transition rate function $b$:
\begin{assumption}
\label{hyp:b-UI}
 We assume that the family
\begin{equation}
	\label{hyp:unifintegr}
	\Big(\frac{1}{n}\int_0^{nt} b(X_s)\,\mathrm ds\Big)_{n\geq 1}
\end{equation}
is uniformly integrable for all $t\geq0$ under $\mathbf{P}_{x}$ for all $x\in E$.
\end{assumption}
\begin{assumption}
We also suppose that for all $i\in I$ and $z\in\mathbb R^d$ we have
\begin{align}
\label{hyp:ErgoF}
\sup_{t>0}\mathbf E_{x_i}\bigg[\Big\lvert\frac{1}{T}\int_t^{t+T}F_i(z,X_{ns})\,\mathrm ds-\bar F_i(z)\Big\rvert\bigg]\underset{n\to\infty}{\longrightarrow}0.
\end{align}
\end{assumption}
\begin{assumption}
\begin{align}
\label{hyp:Fbound}
M_i:=\|F_i(0,\,\cdot\,)\|_\infty=\sup_{x\in E_i}\lvert F_i(0,x)\rvert<\infty.
\end{align}
which implies \eqref{eq:finite-integral-trajectory} for $F_i(0,\,\cdot\,)\in\mathcal{A}_i$ with \eqref{hyp:Lip}.
\end{assumption}
We also remind the Lipschitz assumption \eqref{hyp:Lip}:
\begin{assumption}
For all $i \in I$,
\begin{align*}
L_i := \sup_{x \in E_i} \|F_i(\,\cdot\,,x)\|_{\mathrm{Lip}} < \infty.
\end{align*}
\end{assumption}
\begin{remark} 
\label{rem:UI-implies-L1mu}
Here we remind the result of Lemma 3.8. in \cite{kagan2025averaging}: 
Under  Assumptions~\ref{assumption:ergo} and~\ref{hyp:b-UI}, we have $b \in \mathcal{A}_i$ for all $i \in I$. In particular, for all $x \in E_i$, 
\begin{equation}
\label{eq:ergob}
\lim_{t \to \infty} \frac{1}{t} \int_0^t b(X_s) \mathrm ds = \mu_i(b), \quad \mathbf{P}_x\textup{-a.s.}
\end{equation}
\end{remark}

\subsection{Convergence in law of \texorpdfstring{$\mathbf Z^n$}{Zn}}

In the sequel for $z\in\mathbb R^d$, $n\geq1$ and $i\in I$, we denote by $\mathbf Z^{n,z,i}$ the solution of the following random differential equation
\begin{align*}
\frac{\mathrm d\mathbf Z_t^{n,z,i}}{\mathrm dt}=F_i(\mathbf Z_t^{n,z,i},X_{nt})\quad\textup{with}\quad\mathbf Z_0^{n,z,i}=z,
\end{align*}
and $\mathbf Z^{z,i}$ the solution of the following ordinary differential equation
\begin{align*}
\frac{\mathrm d\bar{\mathbf Z}_t^{z,i}}{\mathrm dt}=\bar F_i(\bar{\mathbf Z}_t^{z,i})\quad\textup{with}\quad\bar{\mathbf Z}_0^{z,i}=z.
\end{align*}

Before stating our theorem we remind the result let us recall the $L^1$ convergence result given in \cite{Hasminskii1966a}: under assumptions \eqref{hyp:Lip} and \eqref{hyp:ErgoF} if $F_i(0,\,\cdot\,)\in\mathcal A_i$ (see \ref{eq:finite-integral-trajectory}), we have the following result
\begin{align}
\label{theo:convSup}
\sup_{t\in[0,T]}\mathbf E_{x_i}\big[\lvert\mathbf Z_t^{n,z,i}-\bar{\mathbf Z}_t^{z,i}\rvert\big]\underset{n\to\infty}{\longrightarrow}0.
\end{align}

We aim to extend the previous theorem to the case where the function $F$ depends on an index process indicating the space in which the process $\mathbf X^n$ evolves.
\begin{theorem}
\label{th:conv}
Let $z\in\mathbb R^d$, $i\in I$, $x_i\in E_i$ and $T>0$, under assumptions \eqref{hyp:Lip}, \eqref{hyp:unifintegr}, \eqref{hyp:ErgoF} and \eqref{hyp:Fbound}, we have, for all $\varphi\in C_b^1(\mathbb R^d)$
\begin{align*}
\mathbf E_{x_i}\big[\varphi(\mathbf Z_T^{n,z})\,\mathbf 1_{\{T\leq\tau_k^n\}}\big]\underset{n\to\infty}{\longrightarrow}\mathbf E_{x_i}\big[\varphi(\bar{\mathbf Z}_T^z)\,\mathbf 1_{\{T\leq\tau_k\}}\big]\quad\forall k\geq1.
\end{align*}
\end{theorem}
We now establish the convergence of the finite-dimensional distributions of the process $\mathbf Z^n$
\begin{theorem}
\label{th:convFD}
Let $N\geq 1$ and $0<t_1<\ldots<t_N$, then, for all $x\in E$, $z\in\mathbb R^d$, $k\geq1$ and $f_1,\ldots,f_N\in C_b^1(\mathbb R^d)$ we have
\begin{align*}
\mathbf E_x\big[f_1(\mathbf Z_{t_1}^{n,z})\cdots f_N(\mathbf Z_{t_N}^{n,z})\,\mathbf 1_{\{t_N\leq\tau_k^n\}}\big]\underset{n\to\infty}{\longrightarrow}\mathbf E_x\big[f_1(\bar{\mathbf Z}_{t_1}^z)\cdots f_N(\bar{\mathbf Z}_{t_N}^z)\,\mathbf 1_{\{t_N\leq\tau_k\}}\big].
\end{align*}
\end{theorem}
Before proving Theorem \ref{th:conv} and \ref{th:convFD} in Section \ref{subsec:proofConv} and \ref{subsec:proofConvFD}, we establish preliminary convergence results in the next section.

\subsection{Convergence in law of \texorpdfstring{$\mathbf Z^n$}{Zn} in Skorokhod topology}

For all $T > 0$, we let $D([0,T], \mathbb R^d)$ denote as usual the space of càdlàg functions defined on $[0,T]$ and with values in $I$. Using the convergence of the finite-dimensional marginals stated in Theorem~\ref{th:convFD}, we prove the convergence of $(\mathbf Z^{n,z}_t)_{t \in[0,T]}$ to $(\bar{\mathbf{Z}}^z_t)_{t \in[0,T]}$ in law for the Skorokhod topology on $D([0,T], \mathbb R^d)$, for all $z\in\mathbb R^d$ under mild assumption on the limit process $\bar{\mathbf{Z}}^z$. 

\medskip

In the following statement, we say that a probability measure $m$ on $\mathbb R^d$ is \emph{tight} if, for all $\varepsilon>0$, there exists a compact  $K_\varepsilon\subset\mathbb R^d$ such that $m(K)\geq 1-\varepsilon$. In particular, if $\bar{\mathbf Z}^z$ is locally compact, all probability measures are tight.
\begin{theorem}
\label{thm:convSkorokhod}
Let $x\in E$, $z\in\mathbb R^d$ and $T>0$. We assume that $\bar{\mathbf Z}^z$ is non-explosive up to time $T$ under $\mathbf P_x$, i.e $\mathbf P_x(\tau_\infty>T)=1$. Then, $(\mathbf Z^{n,z}_t)_{t \in[0,T]}$ converges to $(\bar{\mathbf{Z}}_t^z)_{t \in[0,T]}$ in law for the Skorokhod topology on $D([0,T], I)$, where $D([0,T], I)$ is equipped with the Skorokhod topology induced by the supremum norm.
\end{theorem}
Before turning to the proof of this result, we recall that if $(\mathbf Z_t^{n,z})_{t \in[0,T]}$ converges to $(\mathbf Z_t^z)_{t \in[0,T]}$ in law for the Skorokhod topology on $D([0,T], \mathbb R^d)$ for all $T>0$, then $(\mathbf Z_t^{n,z})_{t \geq 0}$ converges to $(\mathbf{Z}_t^z)_{t \geq 0}$ in law for the Skorokhod topology on $D([0,+\infty), \mathbb R^d)$. As a consequence, if $\mathbf Z^z$ is non-explosive, i.e $\mathbf P_i( \tau_{\infty} = + \infty) = 1$ and if  for all $t \geq 0$, the law of $\mathbf Z_t^z$ is tight under $\mathbb P_i$ for all $i \in I$, then  $(\mathbf Z_t^{n,z})_{t \geq 0}$ converges to $(\mathbf Z_t^z)_{t \geq 0}$ in law for the Skorokhod topology on $D([0,+\infty), \mathbb R^d)$.

\begin{proof}
By Theorem \ref{th:convFD}, the finite-dimensional distributions of $\mathbf Z^{n,z}$ up to the $k$-th jump time converge in law to those of $\bar{\mathbf Z}^z$ up to the $k$-th jump time. Since $\mathbf 1_{\{\tau_k>T\}}\to1$ when $k\to\infty$ and the processes $\mathbf Z^{n,z}$ have continuous sample paths, it remains to establish the tightness of $\mathbf Z^{n,z}$ in $D([0,T],\mathbb R^d)$. In order to so, we make use of Theorem~7.2, p.128 in~\cite{EthierKurtz2005}. This result makes use of the following definition: for all $\delta>0$,
\[
w'(\mathbf Z^{n,z},\delta):=\inf_{\{t_i\}}\max_{1\leq i\leq k} w_{\mathbf I^n}[t_{i-1},t_i),
\]
with $w_{\mathbf I^n}[t_{i-1},t_i)$ the modulus of continuity of $I^n$ over $[t_{i-1},t_i)$ and the $\{t_i\}$ range over all partitions of the form $0=t_0<t_1<\cdots <t_{k-1} < T\leq t_k$ with steps  stricktly greater than $\delta$. According to the above cited result, the tightness property is ensured if, for all $\varepsilon>0$ and all $t\in [0,T]$, there exists a compact set $K_{\varepsilon,t}\subset\mathbb R^d$ such that
\begin{align}
    \label{eq:compactcontainment}
    \limsup_{n\to+\infty} \mathbf P(\mathbf Z^{n,z}_t\in K_{\varepsilon,t})\geq 1-\varepsilon
\end{align}
and there exists $\delta>0$ such that
\begin{align}
    \label{eq:propertyiEthierKurtz}
    \limsup_{n\to+\infty} \mathbf P(w'(\mathbf Z^{n,z},\delta)>\varepsilon)\leq \varepsilon.
\end{align}
In order to check the first property, choose $k_{\varepsilon,t}>0$ such that 
\[
\mathbf P\big(\lvert\bar{\mathbf Z}_t^z\rvert\leq k_{\varepsilon,t}\big)\geq 1-\varepsilon/2,
\]
and $k_{\varepsilon}\geq 1$ 
\[
\mathbf P(T\leq \tau_{k_\varepsilon})\geq 1-\varepsilon/2.
\]
Then, according to Theorem~\ref{th:convFD}, 
\[
\lim_{n\to\infty} \mathbf P(\lvert\mathbf Z^{n,z}_t\rvert\leq k_{\varepsilon,t},\,t\leq \tau_{{k_\varepsilon}}^n)\geq\mathbf E\big[\varphi_{k_{\varepsilon,t}}(\bar{\mathbf Z}_t^z)\mathbf 1_{\{t\leq \tau_{k_\varepsilon}\}}\big]\geq 1-\varepsilon,
\]
where $(\varphi_k)_{k\geq0}$ is a non-decreasing sequence of function such that $\varphi_k\in C_b^1(\mathbb R^d)$, $\mathbf 1_{\lvert z\rvert\leq k-1\}}\leq\varphi_k(z)\leq\mathbf 1_{\{\lvert z\rvert\leq z\}}$ and $\varphi_k\underset{k\to\infty}{\longrightarrow}\mathbf 1_{[0,\infty)}$ which proves~\eqref{eq:compactcontainment}.

\medskip

Let us now check~\eqref{eq:propertyiEthierKurtz}.Let $T>0$, we set $n_\delta=\lfloor T/2\delta\rfloor$ and we divide the interval $[0,T]$ into $n_\delta$ subsets of equal size i.e. we take the subdivision $t_i=iT/n_\delta$ for $0\leq i\leq n_\delta$. Then, for all $x\in E$, $z\in\mathbb R^d$ and $k\geq1$
\begin{align*}
\mathbf P_x\big(\sup_{t_{i-1}\leq s<t\leq t_i}\lvert\mathbf Z_s^{n,z}-\mathbf Z_t^{n,z}\rvert\geq\varepsilon&,\,1\leq i\leq n_\delta\big)\\
&=\mathbf P_x\big(\sup_{t_{i-1}\leq s<t\leq t_i}\lvert\mathbf Z_s^{n,z}-\mathbf Z_t^{n,z}\rvert\geq\varepsilon,\,T\leq\tau_k^n,\,1\leq i\leq n_\delta\big)\\
&+\mathbf P_x\big(\sup_{t_{i-1}\leq s<t\leq t_i}\lvert\mathbf Z_s^{n,z}-\mathbf Z_t^{n,z}\rvert\geq\varepsilon,\,\tau_k^n<T,\,1\leq i\leq n_\delta\big)
\end{align*}
For the first term above we have, for all $L,M>0$ 
\begin{align*}
\mathbf P_x\big(\sup_{t_{i-1}\leq s<t\leq t_i}&\lvert\mathbf Z_s^{n,z}-\mathbf Z_t^{n,z}\rvert\geq\varepsilon,\,T\leq\tau_k^n,\,1\leq i\leq n_\delta\big)\\
&\leq \mathbf P_x\big(\sup_{t_{i-1}\leq s<t\leq t_i}\lvert\mathbf Z_s^{n,z}-\mathbf Z_t^{n,z}\rvert\geq\varepsilon,\,T\leq\tau_k^n,\, L_T^n\leq L,\,M_T^n\leq M,\,1\leq i\leq n_\delta\big)\\
&+\mathbf P_x(T\leq\tau_k^n,\,L_T^n>L)+\mathbf P_x(T\leq\tau_k^n,\,M_T^n>M).
\end{align*}
We remind that on the event $\{T\leq\tau_k^n,\,L_t^n\leq L,\,M_T^n\leq M\}$, we have
\begin{align*}
\lvert\mathbf Z_s^{n,z}-\mathbf Z_t^{n,z}\rvert&\leq(t-s)\,\Big(L\,\big(\lvert z\rvert+TM)\,e^{TL}+M\Big)\\
&\leq\frac{C_T}{n_\delta}
\end{align*}
for $t_{i-1}\leq s<t\leq t_i\leq T$ with $C_T:=T\,(L\,(\lvert z\rvert+T\,M)\,e^{TL}+M)$. Then,
\begin{align*}
\mathbf P_x\big(\sup_{t_{i-1}\leq s<t\leq t_i}\lvert&\mathbf Z_s^{n,z}-\mathbf Z_t^{n,z}\rvert\geq\varepsilon,\,T\leq\tau_k^n,\,1\leq i\leq n_\delta\big)\\
&\leq\mathbf P_x(C_T\geq n_\delta\,\varepsilon)+\mathbf P_x(T\leq\tau_k^n,\,L_T^n>L)+\mathbf P_x(T\leq\tau_k^n,\,M_T^n>M).
\end{align*}
We remind that
\begin{align*}
\limsup_{n\to\infty}\mathbf P_x(T\leq\tau_k^n,L_T^n>L)&\leq\sum_{j=0}^k\mathbf P_x(L_{\mathbf I_{\tau_j}}>L),
\end{align*}
and 
\begin{align*}
\limsup_{n\to\infty}\mathbf P_x(T\leq\tau_k^n,\,M_T^n>M)&\leq\sum_{j=0}^k\mathbf P_x(M_{\mathbf I_{\tau_j}}>M).
\end{align*}
Then, for all $L,M>0$
\begin{align*}
\limsup_{n\to\infty}\mathbf P_x\big(\sup_{t_{i-1}\leq s<t\leq t_i}\lvert\mathbf Z_s^{n,z}&-\mathbf Z_t^{n,z}\rvert\geq\varepsilon,\,T\leq\tau_k^n,\,1\leq i\leq n_\delta\big)\\
&\leq\mathbf P_x(C_T\geq n_\delta\,\varepsilon)+\sum_{j=0}^k\big(\mathbf P_x(L_{\mathbf I_{\tau_j}}>L)+\mathbf P_x(M_{\mathbf I_{\tau_j}}>M)\big).
\end{align*}
Since $L_i<\infty$ and $M_i<\infty$ for all $i\in I$, as $L,M$ tends to infinity, we have 
\begin{align*}
\limsup_{n\to\infty}\mathbf P_x\big(\sup_{t_{i-1}\leq s<t\leq t_i}\lvert\mathbf Z_s^{n,z}-\mathbf Z_t^{n,z}\rvert\geq\varepsilon,\,T\leq\tau_k^n,\,1\leq i\leq n_\delta\big)&\leq\mathbf P_x(C_T\geq n_\delta\,\varepsilon).
\end{align*}
Finally
\begin{align*}
\limsup_{n\to\infty}\mathbf P_x\big(\sup_{t_{i-1}\leq s<t\leq t_i}\lvert\mathbf Z_s^{n,z}-\mathbf Z_t^{n,z}\rvert\geq\varepsilon,\,T\leq\tau_k^n,\,1\leq i\leq n_\delta\big)\underset{\delta\to0}{\longrightarrow}0.
\end{align*}
For the second term 
\begin{align*}
\mathbf P_x\big(\sup_{t_{i-1}\leq s<t\leq t_i}\lvert\mathbf Z_s^{n,z}-\mathbf Z_t^{n,z}\rvert\geq\varepsilon,\,\tau_k^n<T,\,1\leq i\leq n_\delta\big)\leq\mathbf P_x(\tau_k^n<T),
\end{align*}
hence, for all $k\geq0$ 
\begin{align*}
\limsup_{n\to\infty}\mathbf P_x\big(\sup_{t_{i-1}\leq s<t\leq t_i}\lvert\mathbf Z_s^{n,z}-\mathbf Z_t^{n,z}\rvert\geq\varepsilon,\,\tau_k^n<T,\,1\leq i\leq n_\delta\big)\leq\mathbf P_x(\tau_k<T),
\end{align*}
we deduce that~\eqref{eq:propertyiEthierKurtz} holds true which concludes the proof of Theorem~\ref{thm:convSkorokhod}.
\end{proof}

\subsection{Explosion/non-explosion properties}
Since the processes considered in the previous section are defined iteratively, they are only defined up to their explosion time $\tau_\infty^n:=\lim_{k\to\infty} \tau_k^n$. Similarly, the process $\mathbf I$ is defined up to its explosion time $\tau_\infty:=\lim_{k\to\infty} \tau_k$. In this section, we investigate the relation between the explosion time of the accelerated processes and of the limiting process.

\begin{proposition}
Let $T>0$ a strictly positive real number and $x_i\in E_i$ where $i\in I$.  Then
\begin{equation}
\label{eq:exploineq}
    \liminf_{n\to\infty}\mathbf E_{x_i}\big[\varphi(\mathbf Z_t^{n,z})\,\mathbf 1_{\{t<\tau_\infty^n\}}\big]\geq\mathbf E_i\big[\varphi(\bar{\mathbf Z}_t^z)\,\mathbf 1_{\{t<\tau_\infty\}}\big]\quad\forall\varphi\in C_b^1(Z).
\end{equation}
If in addition $\mathbf I$ is non-explosive under $\mathbf P_i$, i.e. $\mathbf P_i(\tau_\infty=\infty)=1$, then
\[
\mathbf E_{x_i}\big[\varphi(\mathbf Z_t^{n,z})\,\mathbf 1_{\{t<\tau_\infty^n\}}\big]\underset{n\to\infty}{\longrightarrow}\mathbf E_i\big[\varphi(\bar{\mathbf Z}_t^z)\big]\quad\forall\varphi\in C_b^1(Z).
\]
\end{proposition}
\begin{proof}
Let $\varphi:Z\longrightarrow\mathbb R$ a bounded derivable function and $T>0$ a strictly positive real number. Without loss of generality we can assume that $\varphi$ is positive. For all $k\geq0$, $i\in I$ and $x_i\in E_i$ we have, according to Theorem~\ref{th:conv},
\[
\mathbf E_{x_i}\big[\varphi(\mathbf Z_T^{n,z})\,\mathbf 1_{\{\tau_k^n>T\}}\big]\underset{n\to\infty}{\longrightarrow}\mathbf E_i\big[\varphi(\bar{\mathbf Z}_T^z)\,\mathbf 1_{\{\tau_k>T\}}\big].
\]
Reminding that for all $n\geq1$, the jumps sequence $(\tau_k^n)_{k\geq0}$ is increasing, by the monotone convergence theorem we have both
\[
\mathbf E_{x_i}\big[\varphi(\mathbf Z_T^{n,z})\,\mathbf 1_{\{\tau_k^n>T\}}\big]\underset{k\to\infty}{\longrightarrow}\mathbf E_{x_i}\big[\varphi(\mathbf Z_T^{n,z})\,\mathbf 1_{\{\tau_\infty^n>T\}}\big]\quad\forall n\geq1,
\]
and 
\[
\mathbf E_i\big[\varphi(\bar{\mathbf Z}_T)\,\mathbf 1_{\{\tau_k>T\}}\big]\underset{k\to\infty}{\longrightarrow}\mathbf E_i\big[\varphi(\bar{\mathbf Z}_T^z)\,\mathbf 1_{\{\tau_\infty>T\}}\big].
\]
Then, for all $\varepsilon>0$, there exists $k_\varepsilon\geq0$ such that
\[
\mathbf E_i\big[\varphi(\bar{\mathbf Z}_T^z)\,\mathbf 1_{\{\tau_k>T\}}\big]\geq\mathbf E_i\big[\varphi(\bar{\mathbf Z}_T^z)\,\mathbf 1_{\{\tau_\infty>T\}}\big]-\varepsilon/2\quad\forall k\geq k_\varepsilon,
\]
and 
\begin{align*}
\mathbf E_{x_i}\big[\varphi(\mathbf Z_T^{n,z})\,\mathbf 1_{\{\tau_{k_\varepsilon}^n>T\}}\big]&\geq\mathbf E_i\big[\varphi(\bar{\mathbf Z}_T^z)\,\mathbf 1_{\{\tau_{k_\varepsilon}>T\}}\big]-\varepsilon/2\\
&\geq\mathbf E_i\big[\varphi(\bar{\mathbf Z}_T^z)\,\mathbf 1_{\{\tau_\infty>T\}}\big]-\varepsilon.
\end{align*}
Using again that $(\tau_k^n)_{k\geq0}$ is increasing for all $n\geq1$ we have
\begin{align*}
\mathbf E_{x_i}\big[\varphi(\mathbf Z_T^{n,z})\,\mathbf 1_{\{\tau_{\infty}^n>T\}}\big]&\geq\mathbf E_{x_i}\big[\varphi(\mathbf Z_T^{n,z})\,\mathbf 1_{\{\tau_{k_\varepsilon}^n>T\}}\big]\\
&\geq\mathbf E_i\big[\varphi(\bar{\mathbf Z}_T^z)\,\mathbf 1_{\{\tau_\infty>T\}}\big]-\varepsilon,
\end{align*}
so
\[
\liminf_{n\to\infty}\mathbf E_{x_i}\big[\varphi(\mathbf Z_T^{n,z})\,\mathbf 1_{\{\tau_{\infty}^n>T\}}\big]\geq\mathbf E_i\big[\varphi(\bar{\mathbf Z}_T^z)\,\mathbf 1_{\{\tau_\infty>T\}}\big],
\]
which conclude the proof of~\eqref{eq:exploineq}.
Let us now consider the case where $\mathbf I$ is non-explosive, i.e. $\mathbb P_i(\tau_\infty=\infty)=1$. Let $\varepsilon>0$, there exists $k_\varepsilon$ such that for all $k\geq k_\varepsilon$
\begin{equation}
\label{ineq:proofexplosion}
\Big\lvert\mathbf E_i\big[\varphi(\mathbf Z_T^{n,z})\,\mathbf 1_{\{T>\tau_k\}}-\varphi(\bar{\mathbf Z}_T^z)\,\mathbf 1_{\{T>\tau_\infty\}}\big]\Big\rvert\leq\varepsilon.
\end{equation}
We decompose 
\begin{align*}
\mathbf E_{x_i}\big[\varphi(\mathbf Z_T^{n,z})\,\mathbf 1_{\{\tau_\infty^n>T\}}\big]&=\mathbf E_{x_i}\big[f(\mathbf Z_T^{n,z})\,\mathbf 1_{\{\tau_k^n>T\}}\big]+\mathbf E_{x_i}\big[\varphi(\mathbf Z_T^{n,z})\,\mathbf 1_{\{\tau_\infty^n>T\geq\tau_k^n\}}\big]\\
&\leq\mathbf E_{x_i}\big[\varphi(\mathbf Z_T^{n,z})\,\mathbf 1_{\{\tau_k^n>t\}}\big]+\|\varphi\|_\infty\,\mathbf P_{x_i}\big(\tau_k^n\leq T\big).
\end{align*}
Since
\[
\mathbf P_{x_i}\big(\tau_k^n>T\big)\underset{n\to\infty}{\longrightarrow}\mathbf P_i\big(\tau_k>T\big)\underset{k\to\infty}{\longrightarrow}1,
\]
increasing $k_\varepsilon$ if necessary, there exists $n_\varepsilon$ such that
\[
\mathbf P_{x_i}\big(\tau_{k_\varepsilon}^n>T\big)\geq 1-\varepsilon\quad\forall n\geq n_\varepsilon.
\]
Hence
\begin{align*}
\Big\lvert\mathbf E_{x_i}\big[\varphi(\mathbf Z_T^{n,z})\,\mathbf 1_{\{\tau_\infty^n>T\}}-\mathbf E_{x_i}\big[\varphi(\mathbf Z_T^{n,z})\,\mathbf 1_{\{\tau_{k_\varepsilon}^n>T\}}\big]\Big\rvert&\leq\|\varphi\|_\infty\,\mathbf P_{x_i}\big(\tau_{k_\varepsilon}^n\leq T\big)\\
&\leq\|\varphi\|_\infty\,\varepsilon.
\end{align*}
and, according to Theorem~\ref{th:conv}and up to a change of $n_\epsilon$,
\[
\Big|\mathbf E_{x_i}\big[\varphi(\mathbf Z_T^{n,z})\,\mathbf 1_{\{\tau_{k_\varepsilon}^n>T\}}-\varphi(\bar{\mathbf Z}_T^z)\,\mathbf 1_{\{\tau_{k_\varepsilon}>T\}}\big]\Big|\leq \varepsilon,\ \forall n\geq n_\varepsilon.
\]
By addition and substraction,
\begin{align*}
\Big\lvert\mathbf E_{x_i}\big[\varphi(\mathbf Z_T^{n,z})\,\mathbf 1_{\{\tau_\infty^n>T\}}-\varphi(\bar{\mathbf Z}_T^z)\,\mathbf 1_{\{\tau_\infty>T\}}\big]\Big\rvert&\leq\Big\lvert\mathbf E_{x_i}\big[\varphi(\mathbf Z_T^{n,z})\,\mathbf 1_{\{\tau_\infty^n>T\}}-\varphi(\mathbf Z_T^{n,z})\,\mathbf 1_{\tau_{k_\varepsilon}^n>T\}}\big]\Big\rvert\\
&+\Big\lvert\mathbf E_{x_i}\big[\varphi(\mathbf Z_T^{n,z})\,\mathbf 1_{\{\tau_{k_\varepsilon}^n>T\}}-\varphi(\bar{\mathbf Z}_T^z)\,\mathbf 1_{\{\tau_k>T\}}\big]\Big\rvert\\
&+\Big\lvert\mathbf E_{x_i}\big[\varphi(\bar{\mathbf Z}_T^z)\,\mathbf 1_{\{\tau_{k}>T\}}-\varphi(\bar{\mathbf Z}_T^z)\,\mathbf 1_{\{\tau_\infty>T\}}\big]\Big\rvert
\end{align*}
so, using \eqref{ineq:proofexplosion}, we have
\[
\limsup_{n\to\infty}\Big\lvert\mathbf E_{x_i}\big[\varphi(\mathbf Z_T^{n,z})\,\mathbf 1_{\{\tau_\infty^n>T\}}-\varphi(\bar{\mathbf Z}_T^z)\,\mathbf 1_{\{\tau_\infty>T\}}\big]\Big\rvert\leq\big(2+\|\varphi\|_\infty\big)\,\varepsilon.
\]
Finally 
\[
\mathbf E_{x_i}\big[\varphi(\mathbf Z_T^{n,z})\,\mathbf 1_{\{\tau_\infty^n>T\}}\big]\underset{n\to\infty}{\longrightarrow}\mathbf E_i\big[\varphi(\bar{\mathbf Z}_T^z)\,\mathbf 1_{\{\tau_\infty>T\}}\big]=\mathbf E_i\big[\varphi(\bar{\mathbf Z}_T^z)\big].
\]
\end{proof}

\section{First properties of convergence}

\label{sec:FirstProperties}

Using \eqref{theo:convSup}, we establish the convergence in law of $\mathbf Z^n$ before the first jump of $\mathbf I^n$:
\begin{proposition}
\label{prop:conv1}
Let $z\in\mathbb R^d$, then for all $i\in I$, $x_i\in E_i$ and $T>0$, if assumptions \eqref{hyp:Lip} and \eqref{hyp:ErgoF} are satisfied and $F_i(z,\,\cdot\,)\in\mathcal A_i$, we have
\begin{align*}
\mathbf E_{x_i}\big[\varphi(\mathbf Z_T^{n,z})\,\mathbf 1_{\{T\leq\tau_1^n\}}\big]\underset{n\to\infty}{\longrightarrow}\mathbf E_{x_i}\big[\varphi(\bar{\mathbf Z}_T^z)\,\mathbf 1_{\{T\leq\tau_1\}}\big],\quad\forall\varphi\in C_b^1(Z).
\end{align*}
\end{proposition}
\begin{proof}
Let $z\in\mathbb R^d$, for all $i\in I$, $x_i\in E_i$ and $T>0$ we have
\begin{align*}
&\mathbf E_{x_i}\big[\varphi(\mathbf Z_T^{n,z})\,\mathbf 1_{\{T\leq\tau_1^n\}}-\varphi(\bar{\mathbf Z}_T^z)\,\mathbf 1_{\{T\leq\tau_1\}}\big]\\
&=\mathbf E_{x_i}\big[\big(\varphi(\mathbf Z_T^{n,z})-\varphi(\bar{\mathbf Z}_T^z)\big)\,\mathbf 1_{\{T\leq\tau_1^n\wedge\tau_1\}}\big]+\mathbf E_{x_i}\big[\varphi(\mathbf Z_T^{n,z})\,\mathbf 1_{\{\tau_1<T\leq\tau_1^n\}}\big]-\mathbf E_{x_i}\big[\varphi(\bar{\mathbf Z}_T^z)\,\mathbf 1_{\{\tau_1^n<T\leq\tau_1\}}\big]\\
&\leq\|\nabla\varphi\|_\infty\,\mathbf E_{x_i}\big[\lvert\mathbf Z_T^{n,z,i}-\bar{\mathbf Z}_T^{z,i}\rvert\,\mathbf 1_{\{T\leq\tau_1^n\wedge\tau_1\}}\big]+\|\varphi\|_\infty\,\big(\mathbf P_{x_i}(\tau_1<T\leq\tau_1^n)+\mathbf P_{x_i}(\tau_1^n<T\leq\tau_1)\big)\\
&\leq\|\nabla\varphi\|_\infty\,\sup_{t\in[0,T]}\mathbf E_{x_i}\big[\lvert\mathbf Z_t^{n,z,i}-\bar{\mathbf Z}_t^{z,i}\rvert\big]+\|\varphi\|_\infty\,\mathbf P_{x_i}\big(\tau_1<T\leq\tau_1^n\big)+\|\varphi\|_\infty\,\mathbf P_{x_i}\big(\tau_1^n<T\leq\tau_1\big).
\end{align*}
By \eqref{theo:convSup} the first term above goes to zero when $n$ goes to infinity. Using the assumption \eqref{hyp:unifintegr} and Remark \ref{rem:UI-implies-L1mu}, $\mathbf P_{x_i}(\tau_1<T\leq\tau_1^n)$ and $\mathbf P_{x_i}(\tau_1^n<T\leq\tau_1)$ go to zero when $n$ goes to infinity and conclues the proof on the proposition. Indeed, using the definition of $\tau_1$ and $\tau_1^n$, we can rewrite
\begin{align*}
\mathbf P_{x_i}(\tau_1<T\leq\tau_1^n)&=\mathbf P_{x_i}\Big(G\Big(\int_0^Tb(X_{ns})\,\mathrm ds\Big)\leq U<G\big(T\mu_i(b)\big)\Big)\\
&=\mathbf E_{x_i}\Big[\Big(G\big(T\mu_i(b)\big)-G\Big(\int_0^Tb(X_{ns})\,\mathrm ds\Big)\Big)_+\Big].
\end{align*}
Then, by assumption \eqref{hyp:unifintegr}, \eqref{eq:ergob} and the Dominated Convergence Theorem the term above goes to zero when $n$ goes to infinity. By the same argument $\mathbf P_{x_i}(\tau_1^n<T\leq\tau_1)$ goes to zero when $n$ goes to infinty  which concludes the proof of the Proposition. 
\end{proof}
Now, we show the convergence of $\mathbf Z^n$ at the first jump-time $\tau_1^n$ in $L^1$
\begin{proposition}
\label{prop:ConvTemps1L1}
For all $i\in I$, $x_i\in E_i$, $z\in\mathbb R^d$ and $T>0$, if assumptions \eqref{hyp:Lip} and \eqref{hyp:ErgoF} hold true and $F_i(z,\,\cdot\,)\in\mathcal A_i$, we have
\begin{align*}
\mathbf E_{x_i}\big[\lvert\mathbf Z_{\tau_1^n}^{n,z}-\bar{\mathbf Z}_{\tau_1}^z\rvert\,\mathbf 1_{\{\tau_1^n\vee\tau_1\leq T\}}\big]\underset{n\to\infty}{\longrightarrow}0.
\end{align*}
\end{proposition}
\begin{proof}
Let $i\in I$, $x_i\in E_i$, $z\in\mathbb R^d$ and $T>0$, 
\begin{align*}
\mathbf E_{x_i}\big[\lvert\mathbf Z_{\tau_1^n}^{n,z}-\bar{\mathbf Z}_{\tau_1}^z\rvert\,\mathbf 1_{\{\tau_1^n\vee\tau_1\leq T\}}\big]&=\mathbf E_{x_i}\big[\lvert\mathbf Z_{\tau_1^n}^{n,z,i}-\bar{\mathbf Z}_{\tau_1}^{z,i}\rvert\,\mathbf 1_{\{\tau_1^n\vee\tau_1\leq T\}}\big].
\end{align*}
For $\tau_1^n\vee\tau_1\leq T$ we have
\begin{align*}
\lvert\mathbf Z_{\tau_1^n}^{n,z,i}-\bar{\mathbf Z}_{\tau_1}^{z,i}\rvert&\leq\big\lvert\mathbf Z_{\tau_1^n}^{n,z,i}-\bar{\mathbf Z}_{\tau_1^n}^{z,i}\big\rvert+\big\lvert\bar{\mathbf Z}_{\tau_1^n}^{z,i}-\bar{\mathbf Z}_{\tau_1}^{z,i}\big\rvert.
\end{align*}
Moreover
\begin{align*}
\lvert\bar{\mathbf Z}_{\tau_1^n}^{z,i}-\bar{\mathbf Z}_{\tau_1}^{z,i}\rvert\leq\int_{\tau_1^n\wedge\tau_1}^{\tau_1^n\vee\tau_1}\lvert\bar F_i(\bar{\mathbf Z}_t^{z,i})\rvert\,\mathrm dt
\end{align*}
and
\begin{align*}
\lvert\bar F_i(\bar{\mathbf Z}_t^{z,i})\rvert&\leq L_i\,\lvert\bar{\mathbf Z}_t^{z,i}\rvert+\lvert\bar F_i(0)\rvert\\
&\leq L_i\,\lvert z\rvert+\lvert\bar F_i(0)\rvert+\int_0^tL_i\,\lvert\bar F_i(\bar{\mathbf Z}_s^{z,i})\rvert\,\mathrm ds
\end{align*}
so by the Gronwall Lemma
\begin{align*}
\lvert\bar F_i(\bar{}\mathbf Z_t^{z,i})\rvert\leq\big(L_i\,\lvert z\rvert+\lvert\bar F_i(0)\rvert\big)\,e^{L_it}.
\end{align*}
Hence, for $\tau_1^n\vee\tau_1\leq T$
\begin{align*}
\lvert\bar{\mathbf Z}_{\tau_1^n}^{z,i}-\bar{\mathbf Z}_{\tau_1}^{z,i}\rvert&\leq \lvert\tau_1^n-\tau_1\rvert\,\big(L_i\,\lvert z\rvert+\lvert\bar F_i(0)\rvert\big)\,e^{L_iT}.
\end{align*}
Finally, for $\tau_1^n\vee\tau_1\leq T$ we have
\begin{align*}
\lvert\mathbf Z_{\tau_1^n}^{n,z,i}-\bar{\mathbf Z}_{\tau_1}^{z,i}\rvert&\leq\big\lvert\mathbf Z_t^{n,z,i}-\bar{\mathbf Z}_t^{z,i}\big\rvert+\lvert\tau_1^n-\tau_1\rvert\,\big(L_i\,\lvert z\rvert+\lvert\bar F_i(0)\rvert\big)\,e^{LT}.
\end{align*}
Therefore
\begin{align*}
\mathbf E_{x_i}\big[\lvert\mathbf Z_{\tau_1^n}^{n,z,i}-\bar{\mathbf Z}_{\tau_1}^{z,i}\rvert\,\mathbf 1_{\{\tau_1^n\vee\tau_1\leq T\}}\big]&\leq\sup_{t\in[0,T]}\mathbf E_{x_i}\big[\big\lvert\mathbf Z_t^{n,z,i}-\bar{\mathbf Z}_t^{z,i}\big\rvert\big]\\
&+C_T\,\mathbf E_{x_i}\big[\lvert\tau_1^n-\tau_1\rvert\,\mathbf 1_{\{\tau_1^n\vee\tau_1\leq T\}}\big].
\end{align*}
By \eqref{theo:convSup} the first term above goes to zero when $n$ goes to infinity. For the second term we use that $\tau_1^n\rightarrow\tau_1$ when $n\rightarrow\infty$ almost-surely and the Dominated Convergence Theorem to get 
\begin{align*}
\mathbf E_{x_i}\big[\lvert\tau_1^n-\tau_1\rvert\,\mathbf 1_{\{\tau_1^n\vee\tau_1\leq T\}}\big]\underset{n\to\infty}{\longrightarrow}0.
\end{align*}
Then we show that 
\begin{align*}
\mathbf E_{x_i}\big[\lvert\mathbf Z_{\tau_1^n}^{n,z,i}-\bar{\mathbf Z}_{\tau_1}^{z,i}\rvert\,\mathbf 1_{\{\tau_1^n\vee\tau_1\leq T\}}\big]\underset{n\to\infty}{\longrightarrow}0,
\end{align*}
and concludes the proof.
\end{proof}
Finally, we prove the convergence in law of $\mathbf Z^n$ at random time $\tau_1^n$
\begin{proposition}
\label{prop:convTemps1}
Let $z\in\mathbb R^d$, then for all $i\in I$, $x_i\in E_i$ and $T>0$, under assumptions \eqref{hyp:Lip} and \eqref{hyp:ErgoF} and if $F_i(z,\,\cdot\,)\in\mathcal A_i$ we have
\begin{align*}
\mathbf E_{x_i}\big[\varphi(\mathbf Z_{\tau_1^n}^{n,z})\,\mathbf 1_{\{\tau_1^n\leq T\}}\big]\underset{n\to\infty}{\longrightarrow}\mathbf E_{x_i}\big[\varphi(\bar{\mathbf Z}_{\tau_1}^z)\,\mathbf 1_{\{\tau_1\leq T\}}\big],\quad\forall\varphi\in C_b^1.
\end{align*}
\end{proposition}
\begin{proof}
Let $z\in\mathbb R^d$, for all $i\in I$, $x_i\in E_i$ and $T>0$ we have
\begin{align*}
&\mathbf E_{x_i}\big[\varphi(\mathbf Z_{\tau_1^n}^{n,z})\,\mathbf 1_{\{\tau_1^n\leq T\}}\big]-\mathbf E_{x_i}\big[\varphi(\bar{\mathbf Z}_{\tau_1}^z)\,\mathbf 1_{\{\tau_1\leq T\}}\big]\\
&=\mathbf E_{x_i}\Big[\big(\varphi(\mathbf Z_{\tau_1^n}^{n,z})-\varphi(\bar{\mathbf Z}_{\tau_1^n}^z)\big)\,\mathbf 1_{\{\tau_1^n\leq T\}}\Big]+\mathbf E_{x_i}\big[\varphi(\bar{\mathbf Z}_{\tau_1^n}^z)\,\mathbf 1_{\{\tau_1^n\leq T\}}-\varphi(\bar{\mathbf Z}_{\tau_1}^z)\,\mathbf 1_{\{\tau_1\leq T\}}\big]\\
&\leq\|\nabla\varphi\|_\infty\,\sup_{t\in[0,T]}\mathbf E_{x_i}\big[\lvert\mathbf Z_t^{n,z,i}-\bar{\mathbf Z}_t^{z,i}\rvert\big]+\|\nabla\varphi\|_\infty\,\mathbf E_{x_i}\big[\lvert\bar{\mathbf Z}_{\tau_1^n}^z-\bar{\mathbf Z}_{\tau_1}\rvert\,\mathbf 1_{\{\tau_1^n\vee\tau_1\leq T\}}\big]\\
&+\|\varphi\|_\infty\,\mathbf P_{x_i}\big(\tau_1^n\leq T<\tau_1\big)+\|\varphi\|_\infty\,\mathbf P_{x_i}\big(\tau_1\leq T<\tau_1^n\big).
\end{align*}
By \eqref{theo:convSup} and the previous Proposition both of the first terms above go to zero when $n$ goes to infinity. As in the proof of the Proposition \ref{prop:conv1} the two last terms go to zero when $n$ goes to infinity.
\end{proof}
\begin{remark}
Under suitable coupling assumptions, it is possible to consider a family $(X^z)_{z\in\mathbb R^d}$ that satisfies an ergodic theorem for fixed $z\in\mathbb R^d$ with invariant measure $\mu_i^z$ on each $E_i$, and to obtain convergence, before the first jump time and at the first jump time $\tau_1^n$ of $\mathbf Z^{z,n}$, toward the averaged process $\bar{\mathbf Z}^z$ satisfying the differential equation
\[
\frac{\mathrm d\bar{\mathbf Z}_t^z}{\mathrm dt}=\bar F_i(\bar{\mathbf Z}_t^z),\quad\bar{\mathbf Z}_0^z=z
\]
with
\[
\bar F_i(z)=\int_{E_i}F_i(z,x)\,\mu_i^z(\mathrm dx)
\]
so as to allow the fast process to depend on the slow process.
\end{remark}

\section{Proof of the Theorem \ref{th:conv}}
\subsection{Technical Lemmas}
To prove our main result, we first establish several lemmas, starting with a Lipschitz property in mean with respect to the initial condition of $\mathbf Z^n$.
\begin{lemma}
\label{lem:majoration1}
For all $n,k\geq1$, $T>0$, $x\in E$ and $y\in\mathbb R^d$ we set
\begin{align*}
\phi_T^{n,k}(y,x)=\mathbf E_{x}\big[\varphi(\mathbf Z_T^{n,y})\,\mathbf 1_{\{T\leq\tau_k^n\}}\big].
\end{align*}
Let $L>0$, then for all $n,k\geq1$, $T>0$ there exist $M_T^k>0$ such that for all $z,y\in\mathbb R^d$ and $x\in E$
\begin{align*}
\lvert\phi_T^{n,k}(y,x)-\phi_T^{n,k}(z,x)\rvert&\leq2\,\|\varphi\|_\infty\,\mathbf P_x(T\leq\tau_k^n,L_T^n>L)+\|\nabla\varphi\|_\infty\,\lvert y-z\rvert\,e^{TL}
\end{align*}
where for all $t>0$ and $n\geq1$
\begin{align*}
L_t^n:=\sup_{s\in[0,t]}L_{\mathbf I_s^n}.
\end{align*}
\end{lemma}
\begin{proof}
Let, $L>0$, for all $n,k\geq1$, $T>0$, $x\in E$ and $y,z\in\mathbb R^d$ we have
\begin{align*}
\lvert\phi_T^{n,k}(y,x)-\phi_T^{n,k}(z,x)\rvert&\leq\mathbf E_x\big[\big\lvert\varphi(\mathbf Z_T^{n,y})-\varphi(\mathbf Z_T^{n,y})\big\rvert\,\mathbf 1_{\{T\leq\tau_k^n\}}\,\mathbf 1_{\{L_T^n>L\}}\big]\\
&+\mathbf E_x\big[\big\lvert\varphi(\mathbf Z_T^{n,y})-\varphi(\mathbf Z_T^{n,y})\big\rvert\,\mathbf 1_{\{T\leq\tau_k^n\}}\,\mathbf 1_{\{L_T^n\leq L\}}\big].
\end{align*}
The first term above is dominated by
\begin{align*}
2\,\|\varphi\|_\infty\,\mathbf P_x(T\leq\tau_k^n,L_T^n>L).
\end{align*}
For the second term we have
\begin{align*}
\mathbf E_x\big[\big(\varphi(\mathbf Z_T^{n,y})-\varphi(\mathbf Z_T^{n,y})\big)\,\mathbf 1_{\{T\leq\tau_k^n\}}\,\mathbf 1_{\{L_T^n\leq L\}}\big]
\leq\|\nabla\varphi\|_\infty\,\mathbf E_x\big[\lvert\mathbf Z_T^{n,y}-\mathbf Z_T^{n,z}\rvert\,\mathbf 1_{\{T\leq\tau_k^n\}}\,\mathbf 1_{\{L_T^n\leq L\}}\big]
\end{align*}
and
\begin{align*}
\lvert\mathbf Z_T^{n,y}-\mathbf Z_T^{n,z}\rvert&\leq\lvert y-z\rvert+\int_0^T\lvert F_{\mathbf I_t^n}(\mathbf Z_t^{n,y},\mathbf X_t^n)-F_{\mathbf I_t^n}(\mathbf Z_t^{n,z},\mathbf X_t^n)\rvert\,\mathrm dt\\
&\leq\lvert y-z\rvert+\int_0^TL_{\mathbf I_t^n}\,\lvert\mathbf Z_t^{n,y}-\mathbf Z_t^{n,z}\rvert\,\mathrm dt\\
&\leq\lvert y-z\rvert+L_T^n\int_0^T\lvert\mathbf Z_t^{n,y}-\mathbf Z_t^{n,z}\rvert\,\mathrm dt
\end{align*}
So, by the Gronwall Lemma we have
\begin{align*}
\lvert\mathbf Z_T^{n,y}-\mathbf Z_T^{n,z}\rvert&\leq\lvert y-z\rvert\,e^{TL_T^n}.
\end{align*}
Then, 
\begin{align*}
\mathbf E_x\big[\lvert\mathbf Z_T^{n,y}-\mathbf Z_T^{n,y}\rvert\,\mathbf 1_{\{T\leq\tau_k^n\}}\,\mathbf 1_{\{L_T^n\leq L\}}\big]\leq\lvert y-z\rvert\,e^{TL}\,\mathbf P_x(L_T^n\leq L).
\end{align*}
Finally 
\begin{align*}
\lvert\phi_T^{n,k}(y,x)-\phi_T^{n,k}(z,x)\rvert\leq2\,\|\varphi\|_\infty\,\mathbf P_x(T\leq\tau_k^n,L_T^n>L)+\|\nabla\varphi\|_\infty\,\lvert y-z\rvert\,e^{TL}\,\mathbf P_x(L_T^n\leq L),
\end{align*}
and conclues the proof of the Lemma.
\end{proof}
Now we get a uniform bound on $\mathbf Z^n$.
\begin{lemma}
\label{lem:majorationZn}
Let $n\geq1$ and $z\in\mathbb R^d$, then for all $0<t<T$, if \eqref{hyp:Fbound} and \eqref{hyp:Lip} are satisfied we have
\begin{align*}
\lvert\mathbf Z_t^{n,z}\rvert\leq\big(\lvert z\rvert+T\,M_T^n)\,e^{TL_T^n},
\end{align*}
where for all $t>0$ and $n\geq1$
\begin{align*}
M_t^n:=\sup_{s\in[0,t]}M_{\mathbf I_s^n}.
\end{align*}
\end{lemma}
\begin{proof}
For all $i\in I$ we set
Let $n\geq1$ and $z\in\mathbb R^d$, for all $0<t<T$ we have
\begin{align*}
\lvert\mathbf Z_t^{n,z}\rvert&\leq\lvert z\rvert+\Big\lvert\int_0^t F_{\mathbf I_s^n}(\mathbf Z_s^{n,z},\mathbf X_s^n)\,\mathrm ds\Big\rvert\\
&\leq\lvert z\rvert+\int_0^tL_{\mathbf I_s^n}\,\lvert\mathbf Z_s^{n,z}\rvert\,\mathrm ds+\int_0^t\lvert F_{\mathbf I_s^n}(0,\mathbf X_s^n)\rvert\,\mathrm ds\\
&\leq\lvert z\rvert+L_t^n\int_0^t\lvert\mathbf Z_s^{n,z}\rvert\,\mathrm ds+t\,M_t^n\\
&\leq\lvert z\rvert+L_T^n\int_0^t\lvert\mathbf Z_s^{n,z}\rvert\,\mathrm ds+T\,M_T^n
\end{align*}
So by the Gronwall Lemma we have
\begin{align*}
\lvert\mathbf Z_t^{n,z}\rvert\leq\big(\lvert z\rvert+T\,M_T^n)\,e^{TL_T^n},
\end{align*}
which conclues the proof of the Lemma.
\end{proof}
Finally, we establish another Lipschitz property in mean for $\mathbf Z^n$, this time with respect to time.
\begin{lemma}
\label{lem:majoration2}
We recall the notation from Lemma \ref{lem:majoration1}: for all $n,k\geq1$, $t>0$, $x\in E$ and $y\in\mathbb R^d$ we set
\begin{align*}
\phi_t^{n,k}(y,x)=\mathbf E_x\big[\varphi(\mathbf Z_t^{n,y})\,\mathbf 1_{\{t\leq\tau_k^n\}}\big].
\end{align*}
Then, for all $n,k\geq1$, $0<s<t<T$ $x\in E$ and $z\in\mathbb R^d$, under assumptions \eqref{hyp:Lip} and \eqref{hyp:Fbound} we have
\begin{align*}
\lvert\phi_t^{n,k}(z,x)-\phi_s^{n,k}(z,x)\rvert&\leq(t-s)\,\|\nabla\varphi\|_\infty\,\Big(L\,\big(\lvert z\rvert+TM)\,e^{TL}+M\Big)\\
&+2\,\|\varphi\|_\infty\,\big(\,\mathbf P_x(t\leq\tau_k^n,L_t^n>L)+\,\mathbf P_x(t\leq\tau_k^n,M_t^n>M)\big)\\
&+\|\varphi\|_\infty\,\mathbf P_x(s\leq\tau_k^n<t).
\end{align*}
\end{lemma}
\begin{proof}
Let $n,k\geq 1$, $x\in E$ and $z\in\mathbb R^d$, then for all $0<s<t<T$ we have
\begin{align*}
\lvert\phi_t^{n,k}(z,x)-\phi_s^{n,k}(z,x)\rvert&\leq\mathbf E_x\big[\lvert\varphi(\mathbf Z_t^{n,z})\,\mathbf 1_{\{t\leq\tau_k^n\}}-\varphi(\mathbf Z_s^{n,z})\,\mathbf 1_{\{s\leq\tau_k^n\}}\rvert\big]\\
&=\mathbf E_x\Big[\big\lvert\varphi(\mathbf Z_t^{n,z})-\varphi(\mathbf Z_s^{n,z})\big\rvert\,\mathbf 1_{\{t\leq\tau_k^n\}}\Big]+\mathbf E_x\big[\lvert\varphi(\mathbf Z_s^{n,z})\rvert\,\mathbf 1_{\{s\leq\tau_k^n<t\}}\big]\\
&\leq\mathbf E_x\Big[\big\lvert\varphi(\mathbf Z_t^{n,z})-\varphi(\mathbf Z_s^{n,z})\big\rvert\,\mathbf 1_{\{t\leq\tau_k^n\}}\Big]+\|\varphi\|_\infty\,\mathbf P_x(s\leq\tau_k^n<t).
\end{align*}
For all $L,M>0$, we have
\begin{align*}
\mathbf E_x\Big[\big\lvert\varphi(\mathbf Z_t^{n,z})-\varphi(\mathbf Z_s^{n,z})\big\rvert\,\mathbf 1_{\{t\leq\tau_k^n\}}\Big]&\leq\mathbf E_x\Big[\big\lvert\varphi(\mathbf Z_t^{n,z})-\varphi(\mathbf Z_s^{n,z})\big\rvert\,\mathbf 1_{\{t\leq\tau_k^n\}}\,\mathbf 1_{\{L_t^n>L\}}\Big]\\
&+\mathbf E_x\Big[\big\lvert\varphi(\mathbf Z_t^{n,z})-\varphi(\mathbf Z_s^{n,z})\big\rvert\,\mathbf 1_{\{t\leq\tau_k^n\}}\,\mathbf 1_{\{M_t^n>M\}}\Big]\\
&+\mathbf E_x\Big[\big\lvert\varphi(\mathbf Z_t^{n,z})-\varphi(\mathbf Z_s^{n,z})\big\rvert\,\mathbf 1_{\{t\leq\tau_k^n\}}\,\mathbf 1_{\{L_t^n\leq L,M_t^n\leq M\}}\Big].
\end{align*}
The two first terms above are bounded above by
\begin{align*}
2\,\|\varphi\|_\infty\,\big(\mathbf P_x(t\leq\tau_k^n,L_t^n>L)+\mathbf P_x(t\leq\tau_k^n,M_t^n>M)\big).
\end{align*}
For the last term we have
\begin{multline*}
\mathbf E_x\Big[\big\lvert\varphi(\mathbf Z_t^{n,z})-\varphi(\mathbf Z_s^{n,z})\big\rvert\,\mathbf 1_{\{t\leq\tau_k^n\}}\,\mathbf 1_{\{L_t^n\leq L,M_t^n\leq M\}}\Big]\\
\leq\|\nabla\varphi\|_\infty\,\mathbf E_x\big[\lvert\mathbf Z_t^{n,z}-\mathbf Z_s^{n,z}\rvert\,\mathbf 1_{\{t\leq\tau_k^n\}}\,\mathbf 1_{\{L_t^n\leq L,M_t^n\leq M\}}\big]
\end{multline*}
We have
\begin{align*}
\mathbf Z_t^{n,z}-\mathbf Z_s^{n,z}&=\int_s^tF_{\mathbf I_u^n}(\mathbf Z_u^{n,z},\mathbf X_u^n)\,\mathrm du\\
&=\int_s^t\big[F_{\mathbf I_u^n}(\mathbf Z_u^{n,z},\mathbf X_u^n)-F_{\mathbf I_u^n}(0,\mathbf X_u^n)\big]\,\mathrm du+\int_s^tF_{\mathbf I_u^n}(0,\mathbf X_u^n)\,\mathrm du.
\end{align*}
So
\begin{align*}
\lvert\mathbf Z_t^{n,z}-\mathbf Z_s^{n,z}\rvert\,\mathbf 1_{\{t\leq\tau_k^n\}}&\leq\int_s^tL_{\mathbf I_u^n}\,\lvert\mathbf Z_u^{n,z}\rvert\,\mathrm du+(t-s)\,\sup_{u\in[s,t]}M_{\mathbf I_u^n}\\
&\leq L_t^n\int_s^t\lvert\mathbf Z_u^{n,z}\rvert\,\mathrm du+(t-s)\,M_t^n,
\end{align*}
using Lemma \ref{lem:majorationZn} we have
\begin{align*}
\lvert\mathbf Z_t^{n,z}-\mathbf Z_s^{n,z}\rvert\leq(t-s)\Big(L_t^n\,\big(\lvert z\rvert+T\,M_t^n)\,e^{TL_t^n}+M_t^n\Big),
\end{align*}
and 
\begin{align*}
\mathbf E_x\big[\lvert\mathbf Z_t^{n,z}-\mathbf Z_s^{n,z}\rvert\,\mathbf 1_{\{t\leq\tau_k^n\}}\,\mathbf 1_{\{L_t^n\leq L,M_t^n\leq M\}}\big]\leq(t-s)\Big(L\,\big(\lvert z\rvert+TM)\,e^{TL}+M\Big)
\end{align*}
which concludes the proof.
\end{proof}

\subsection{Proof of Theorem \ref{th:conv}: convergence in law of \texorpdfstring{$\mathbf Z^n$}{Zn}}
\label{subsec:proofConv}
Having established the necessary preliminary results, we are now ready to prove the theorem.
\begin{proof}[Proof of Theorem \ref{th:conv}]
We make the proof by induction on $k\geq1$. For $k=1$, it is none other than the Proposition \ref{prop:conv1}. Using the notation of Lemma \ref{lem:majoration1}, we assume that we have the following recurrence hypothesis 
\begin{align}
\label{hyp:reccurence2}
\phi_T^{n,k}(z,x_i)\underset{n\to\infty}{\longrightarrow}\phi_T^k(z,x_i):=\mathbf E_{x_i}\big[\varphi(\bar{\mathbf Z}_T^z)\,\mathbf 1_{\{T\leq\tau_k\}}\big],\quad\forall T>0, i\in I\quad\textup{and}\quad x_i\in E_i.
\end{align} 
For $k+1$ we have
\begin{align*}
\mathbf E_{x_i}\big[\varphi(\mathbf Z_T^{n,z})\,\mathbf 1_{\{T\leq\tau_{k+1}^n\}}\big]&=\mathbf E_{x_i}\big[\varphi(\mathbf Z_T^{n,z})\,\mathbf 1_{\{T\leq\tau_1^n\}}\big]+\mathbf E_{x_i}\big[\varphi(\mathbf Z_T^{n,z})\,\mathbf 1_{\{\tau_1^n<T\leq\tau_{k+1}^n\}}\big]
\end{align*}
\textit{Step 1} By the Proposition \ref{prop:conv1} the first term above goes to 
\begin{align*}
\mathbf E_{x_i}\big[\varphi(\bar{\mathbf Z}_T^z)\,\mathbf 1_{\{T\leq\tau_1\}}\big].
\end{align*}
For the second, on the event $\{\tau_1^n<T\leq\tau_{k+1}^n\}$ we have
\begin{align*}
\varphi(\mathbf Z_T^{n,z})&=\varphi(\mathbf Z_{\tau_1^n}^{n,z})+\int_0^{T-\tau_1^n}\nabla\varphi(\mathbf Z_{t+\tau_1^n}^{n,z})\,F_{\mathbf I_{t+\tau_1^n}^n}(\mathbf Z_{t+\tau_1^n}^{n,z},\mathbf X_{t+\tau_1^n}^n)\,\mathrm dt
\end{align*}
so by the Strong Markov property applied to $(\mathbf Z^{n,z},\mathbf X^n)$ at time $\tau_1^n$ we have
\begin{align*}
\mathbf E_{x_i}\big[\varphi(\mathbf Z_T^{n,z})\,\mathbf 1_{\{\tau_1^n<T\leq\tau_{k+1}^n\}}\big]&=\mathbf E_{x_i}\big[\mathbf 1_{\{\tau_1^n<T\}}\,\phi_{T-\tau_1^n}^{n,k}(\mathbf Z_{\tau_1^n}^{n,z},\mathbf X_{\tau_1^n}^n)\big],
\end{align*}
which is equal to
\begin{align*}
\mathbf E_{x_i}\big[\mathbf 1_{\{\tau_1^n\vee\tau_1<T\}}\,\phi_{T-\tau_1^n}^{n,k}(\mathbf Z_{\tau_1^n}^{n,z},\mathbf X_{\tau_1^n}^n)\big]+\mathbf E_{x_i}\big[\mathbf 1_{\{\tau_1^n<T\leq\tau_1\}}\,\phi_{T-\tau_1^n}^{n,k}(\mathbf Z_{\tau_1^n}^{n,z},\mathbf X_{\tau_1^n}^n)\big].
\end{align*}
The second term above is majorated by
\begin{align*}
\|\varphi\|_\infty\,\mathbf P_{x_i}(\tau_1^n<T\leq\tau_1),
\end{align*}
which goes to zero when $n$ goes to infinity.
By addition and substraction, the first term is equal to
\begin{multline*}
\mathbf E_{x_i}\Big[\mathbf 1_{\{\tau_1^n\vee\tau_1<T\}}\,\big(\phi_{T-\tau_1^n}^{n,k}(\mathbf Z_{\tau_1^n}^{n,z},\mathbf X_{\tau_1^n}^n)-\phi_{T-\tau_1^n}^{n,k}(\bar{\mathbf Z}_{\tau_1}^z,\mathbf X_{\tau_1^n}^n)\big)\Big]\\
+\mathbf E_{x_i}\big[\mathbf 1_{\{\tau_1^n\vee\tau_1<T\}}\,\phi_{T-\tau_1^n}^{n,k}(\bar{\mathbf Z}_{\tau_1}^z,\mathbf X_{\tau_1^n}^n)\big].
\end{multline*}
\textit{Step 2} Let $L>0$, using Lemma \ref{lem:majoration1}, we have
\begin{align*}
\lvert(\phi_{T-\tau_1^n}^{n,k}(\mathbf Z_{\tau_1^n}^{n,z},\mathbf X_{\tau_1^n}^n)-\phi_{T-\tau_1^n}^{n,k}(\bar{\mathbf Z}_{\tau_1}^z,\mathbf X_{\tau_1^n}^n)\big\rvert\leq2\,\|\varphi\|_\infty\,\mathbf P_{\mathbf X_{\tau_1^n}^n}(L_{\tau_k^n}^n>L)+\lvert\mathbf Z_{\tau_1^n}^{n,z}-\bar{\mathbf Z}_{\tau_1}^z\rvert\,e^{TL}
\end{align*}
therefore,
\begin{align*}
\mathbf E_{x_i}\Big[\mathbf 1_{\{\tau_1^n\vee\tau_1<T\}}\,\big(\phi_{T-\tau_1^n}^{n,k}(\mathbf Z_{\tau_1^n}^{n,z}&,\mathbf X_{\tau_1^n}^n)-\phi_{T-\tau_1^n}^{n,k}(\bar{\mathbf Z}_{\tau_1}^z,\mathbf X_{\tau_1^n}^n)\big)\Big]\\
&\leq2\,\|\varphi\|_\infty\,\mathbf E_{x_i}\big[\mathbf 1_{\{\tau_1^n<T\}}\,\mathbf P_{\mathbf X_{\tau_1^n}^n}(T-t\leq\tau_k^n,L_{T-t}^n>L)_{t=\tau_1^n}\big]\\
&+\mathbf E_{x_i}\big[\lvert\mathbf Z_{\tau_1^n}^{n,z}-\bar{\mathbf Z}_{\tau_1}^z\rvert\,\mathbf 1_{\{\tau_1^n\vee\tau_1<T\}}\,e^{(T-\tau_1)L}\big]\\
&\leq2\,\|\varphi\|_\infty\,\mathbf E_{x_i}\big[\mathbf 1_{\{\tau_1^n<T\}}\,\mathbf P_{\mathbf X_{\tau_1^n}^n}(T-t\leq\tau_k^n,L_{T-t}^n>L)_{t=\tau_1^n}\big]\\
&+\mathbf E_{x_i}\big[\lvert\mathbf Z_{\tau_1^n}^{n,z}-\bar{\mathbf Z}_{\tau_1}^z\rvert\,\mathbf 1_{\{\tau_1^n\vee\tau_1<T\}}\big]\,e^{TL}.
\end{align*}
For all $L>0$, $n,k\geq 1$ and $x\in E$
\begin{align*}
\mathbf P_x(T-t\leq\tau_k^n,L_{T-t}^n>L)&=\mathbf P_x\big(T-t\leq\tau_k^n,\sup_{s\in[0,T-t]}L_{\mathbf I_t^n}>L\big)\\
&=\mathbf P_x\big(T-t\leq\tau_k^n,\max_{0\leq i\leq k}L_{\mathbf I_{\tau_i^n}^n}>L\big)\\
&\leq\sum_{j=0}^k\mathbf P_x(T-t\leq\tau_k^n,L_{\mathbf I_{\tau_j^n}^n}>L).
\end{align*}
By Remark 3.5. following Corollary 3.3 in \cite{kagan2025averaging} we have for all $x\in E$ and $f$ a bounded measurable function
\begin{align*}
\mathbf E_x\big[f(\mathbf I_{\tau_j^n}^n)\,\mathbf 1_{\{T-t\leq \tau_k^n\}}\big]\underset{n\to\infty}{\longrightarrow}\mathbf E_x\big[f(\mathbf I_{\tau_j})\,\mathbf 1_{\{T-t\leq \tau_k\}}\big],\quad\forall j,k\geq0.
\end{align*}
Hence
\begin{align}
\label{eq:convL}
\mathbf P_x(T-t\leq\tau_k^n,L_{\mathbf I_{\tau_j^n}^n}>L)\underset{n\to\infty}{\longrightarrow}\mathbf P_x(T-t\leq\tau_k,L_{\mathbf I_{\tau_j}}>L).
\end{align}
Now by definition of the process $\mathbf X^n$ we have
\begin{align*}
&\mathbf E_{x_i}\big[\mathbf 1_{\{\tau_1^n<T\}}\,\mathbf P_{\mathbf X_{\tau_1^n}^n}(T-t\leq\tau_k^n,L_{T-t}^n>L)_{t=\tau_1^n}\big]\\
&=\mathbf E_{x_i}\Big[\int_0^T\mathrm dt\,b(X_{nt})\,G'\Big(\int_0^tb(X_{ns})\,\mathrm ds\Big)\int_E\pi(X_{nt},\mathrm dy)\,\mathbf P_y(T-t\leq\tau_k^n,L_{\mathbf I_{\tau_j^n}^n}>L)\Big]\\
&=\mathbf E_{x_i}\Big[\int_0^T\mathrm dt\,b(X_{nt})\,\Big(G'\Big(\int_0^tb(X_{ns})\,\mathrm ds\Big)-G'\big(t\mu_i(b)\big)\Big)\\
&\times\int_E\pi(X_{nt},\mathrm dy)\,\mathbf P_y(T-t\leq\tau_k^n,L_{\mathbf I_{\tau_j^n}^n}>L)\Big]\\
&+\mathbf E_{x_i}\Big[\int_0^T\mathrm dt\,b(X_{nt})\,G'\big(t\mu_i(b)\big)\int_E\pi(X_{nt},\mathrm dy)\,\mathbf P_y(T-t\leq\tau_k^n,L_{\mathbf I_{\tau_j^n}^n}>L)\Big].
\end{align*}
Using the same computation as in the proof of Proposition 3.1 in Section 5.1 of \cite{kagan2025averaging}, the first term above goes to zero as $n$ goes to infinity. By Lemma 5.1 in \cite{kagan2025averaging} and \eqref{eq:convL}, we have
\begin{align*}
\limsup_n\mathbf E_{x_i}\big[\mathbf 1_{\{\tau_1^n<T\}}\,\mathbf P_{\mathbf X_{\tau_1^n}^n}(T-t\leq\tau_k^n,L_{T-t}^n>L)_{t=\tau_1^n}\big]\leq\sum_{j=0}^k\mathbf E_i\big[\mathbf 1_{\{\tau_1<T\}}\,\mathbf P_{\mathbf I_{\tau_1}}(L_{\mathbf I_{\tau_j}}>L)\big].
\end{align*}
By the Proposition \ref{prop:ConvTemps1L1}
\begin{align*}
\mathbf E_{x_i}\big[\lvert\mathbf Z_{\tau_1^n}^{n,z}-\bar{\mathbf Z}_{\tau_1}^z\rvert\,\mathbf 1_{\{\tau_1^n\vee\tau_1<T\}}\big]\underset{n\to\infty}{\longrightarrow}0.
\end{align*}
Combining the inequalities, we have, for all $L>0$
\begin{multline*}
\limsup_n\mathbf E_{x_i}\Big[\mathbf 1_{\{\tau_1^n\vee\tau_1<T\}}\,\big(\phi_{T-\tau_1^n}^{n,k}(\mathbf Z_{\tau_1^n}^{n,z},\mathbf X_{\tau_1^n}^n)-\phi_{T-\tau_1^n}^{n,k}(\bar{\mathbf Z}_{\tau_1}^z,\mathbf X_{\tau_1^n}^n)\big)\Big]\\
\leq2\|\varphi\|_\infty\sum_{j=0}^k\mathbf E_i\big[\mathbf 1_{\{\tau_1<T\}}\,\mathbf P_{\mathbf I_{\tau_1}}(L_{\mathbf I_{\tau_j}}>L)\big]
\end{multline*}
Finally using assumption \eqref{hyp:Lip}, for all $j\geq0$ and $\ell\in I$
\begin{align*}
\mathbf P_\ell(L_{\mathbf I_{\tau_j}}>L)\underset{L\to\infty}{\longrightarrow}0.
\end{align*}
Since the sum is finite we get
\begin{align*}
\mathbf E_{x_i}\Big[\mathbf 1_{\{\tau_1^n\vee\tau_1<T\}}\,\big(\phi_{T-\tau_1^n}^{n,k}(\mathbf Z_{\tau_1^n}^{n,z},\mathbf X_{\tau_1^n}^n)-\phi_{T-\tau_1^n}^{n,k}(\bar{\mathbf Z}_{\tau_1}^z,\mathbf X_{\tau_1^n}^n)\big)\Big]\underset{n\to\infty}{\longrightarrow}0.
\end{align*}
\textit{Step 3} By the definition of the process $\mathbf X^n$, we have
\begin{multline*}
\mathbf E_{x_i}\big[\mathbf 1_{\{\tau_1^n\vee\tau_1<T\}}\,\phi_{T-\tau_1^n}^{n,k}(\bar{\mathbf Z}_{\tau_1}^z,\mathbf X_{\tau_1^n}^n)\big]\\
=\mathbf E_{x_i}\Big[\,\mathbf 1_{\{\tau_1<T\}}\int_0^T\mathrm dt\,b(X_{nt})\,G'\Big(\int_0^tb(X_{ns})\,\mathrm ds\Big)\int_E\pi(X_{nt},\mathrm dy)\,\phi_{T-t}^{n,k}(\bar{\mathbf Z}_{\tau_1}^z,y)\Big],
\end{multline*}
which is equal by addition and substraction to
\begin{multline*}
\mathbf E_{x_i}\Big[\,\mathbf 1_{\{\tau_1<T\}}\int_0^T\mathrm dt\,b(X_{nt})\,\Big(G'\Big(\int_0^tb(X_{ns})\,\mathrm ds\Big)-G'\big(t\mu_i(b)\big)\Big)\int_E\pi(X_{nt},\mathrm dy)\,\phi_{T-t}^{n,k}(\bar{\mathbf Z}_{\tau_1}^z,y)\Big]\\
+\mathbf E_{x_i}\Big[\,\mathbf 1_{\{\tau_1<T\}}\int_0^T\mathrm dt\,b(X_{nt})\,G'\big(t\mu_i(b)\big)\int_E\pi(X_{nt},\mathrm dy)\,\phi_{T-t}^{n,k}(\bar{\mathbf Z}_{\tau_1}^z,y)\Big].
\end{multline*}
Using the same technics as in the proof of the Proposition 2.3. in \cite{kagan2025averaging} the first term above goes to zero when $n$ goes to infinity.\\
\textit{Step 4} For all $N\geq 1$ we have the following decomposition
\begin{multline*}
\int_0^T\mathrm dt\,b(X_{nt})\,G'\big(t\mu_i(b)\big)\int_E\pi(X_{nt},\mathrm dy)\,\phi_{T-t}^{n,k}(\bar{\mathbf Z}_{\tau_1}^z,y)\\
=\sum_{j=0}^{N-1}\int_{jT/N}^{(j+1)T/N}\mathrm dt\,b(X_{nt})\,G'\big(t\mu_i(b)\big)\int_E\pi(X_{nt},\mathrm dy)\,\phi_{T-t}^{n,k}(\bar{\mathbf Z}_{\tau_1}^z,y).
\end{multline*}
By addition and substraction we have
\begin{align*}
&\int_{jT/N}^{(j+1)T/N}\mathrm dt\,b(X_{nt})\,G'\big(t\mu_i(b)\big)\int_E\pi(X_{nt},\mathrm dy)\,\phi_{T-t}^{n,k}(\bar{\mathbf Z}_{\tau_1}^z,y)\\
&=\int_{jT/N}^{(j+1)T/N}\mathrm dt\,b(X_{nt})\,G'\big(t\mu_i(b)\big)\int_E\pi(X_{nt},\mathrm dy)\,\big(\phi_{T-t}^{n,k}(\bar{\mathbf Z}_{\tau_1}^z,y)-\phi_{T-jT/N}^{n,k}(\bar{\mathbf Z}_{\tau_1}^z,y)\big)\\
&+\int_{jT/N}^{(j+1)T/N}\mathrm dt\,b(X_{nt})\,G'\big(t\mu_i(b)\big)\int_E\pi(X_{nt},\mathrm dy)\,\phi_{T-jT/N}^{n,k}(\bar{\mathbf Z}_{\tau_1}^z,y).
\end{align*}
By Lemma \ref{lem:majoration2}, for all $L,M>0$ we have
\begin{align*}
\lvert\phi_{T-t}^{n,k}(\bar{\mathbf Z}_{\tau_1}^z,y)-\phi_{T-jT/N}^{n,k}(\bar{\mathbf Z}_{\tau_1}^z,y)\rvert&\leq(t-jT/N)\,\Big(L\,\big(\lvert\bar{\mathbf Z}_{\tau_1}^z\rvert+T\,M\big)\,e^{TL}+M\Big)\\
&+\|\varphi\|_\infty\,\mathbf P_y(T-t\leq\tau_k^n,L_{T-t}^n>L)\\
&+\|\varphi\|_\infty\,\mathbf P_y(T-t\leq\tau_k^n,M_{T-t}^n>M)\\
&+\|\varphi\|_\infty\,\mathbf P_y\big(T-(j+1)T/N\leq\tau_k^n<T-jT/N\big)\\
&\leq\frac{T}{N}\,\Big(L\,\big(\lvert\bar{\mathbf Z}_{\tau_1}^z\rvert+T\,M\big)\,e^{TL}+M\Big)\\
&+\|\varphi\|_\infty\,\mathbf P_y(T-t\leq\tau_k^n,L_{T-t}^n>L)\\
&+\|\varphi\|_\infty\,\mathbf P_y(T-t\leq\tau_k^n,M_{T-t}^n>M)\\
&+\|\varphi\|_\infty\,\mathbf P_y\big(T-(j+1)T/N\leq\tau_k^n<T-jT/N\big).
\end{align*}
Then 
\begin{align*}
&\Big\lvert\sum_{j=0}^{N-1}\int_{jT/N}^{(j+1)T/N}\mathrm dt\,b(X_{nt})\,G'\big(t\mu_i(b)\big)\int_E\pi(X_{nt},\mathrm dy)\,\big(\phi_{T-t}^{n,k}(\bar{\mathbf Z}_{\tau_1}^z,y)-\phi_{T-jT/N}^{n,k}(\bar{\mathbf Z}_{\tau_1}^z,y)\big)\Big\rvert\\
&\leq\frac{T}{N}\,\Big(L\,\big(\lvert\bar{\mathbf Z}_{\tau_1}^z\rvert+T\,M\big)\,e^{TL}+M\Big)\,G'\big(T\mu_i(b)\big)+\|\varphi\|_\infty\,\sum_{j=0}^{N-1}\int_{jT/N}^{(j+1)T/N}\mathrm dt\,b(X_{nt})\\
&\times G'\big(t\mu_i(b)\big)\int_E\pi(X_{nt},\mathrm dy)\,\big(\mathbf P_y(T-t\leq\tau_k^n,L_{T-t}^n>L)+\mathbf P_y(T-t\leq\tau_k^n,M_{T-t}^n>M)\\
&+\mathbf P_y(T-(j+1)T/N\leq\tau_k^n<T-jT/N)\big).
\end{align*}
As in the step $2$ of this proof, we have
\begin{align*}
\mathbf P_y(T-t\leq\tau_k^n,L_{T-t}^n>L)\leq\sum_{\ell=0}^k\mathbf P_y(T-t\leq\tau_k^n,L_{\mathbf I_{\tau_\ell^n}^n}>L),
\end{align*}
and 
\begin{align*}
\mathbf P_y(T-t\leq\tau_k^n,L_{T-t}^n>L)\leq\sum_{\ell=0}^k\mathbf P_y(T-t\leq\tau_k^n,L_{\mathbf M_{\tau_\ell^n}^n}>M).
\end{align*}
So, using Lemma 5.1. in \cite{kagan2025averaging} as in the step 2 the following term
\begin{multline*}
\mathbf E_{x_i}\Big[\int_{jT/N}^{(j+1)T/N}\mathrm dt\,b(X_{nt})\,G'\big(t\mu_i(b)\big)\int_E\pi(X_{nt},\mathrm dy)\,\Big(\sum_{\ell=0}^k\mathbf P_y(T-t\leq\tau_k^n,L_{\mathbf I_{\tau_\ell^n}^n}>L)\\
+\sum_{\ell=0}^k\mathbf P_y(T-t\leq\tau_k^n,M_{\mathbf I_{\tau_\ell^n}^n}>M)+\mathbf P_y(T-(j+1)T/N\leq\tau_k^n<T-jT/N)\Big)\Big]
\end{multline*}
goes to 
\begin{multline*}
\int_{jT/N}^{(j+1)T/N}\mathrm dt\,G'\big(t\mu_i(b)\big)\int_E\mu_i(\mathrm dx)\,b(x)\int_E\pi(x,\mathrm dy)\,\Big(\sum_{\ell=0}^k\mathbf P_y(T-t\leq\tau_k,L_{\mathbf I_{\tau_\ell}}>L)\\
+\sum_{\ell=0}^k\mathbf P_y(T-t\leq\tau_k,M_{\mathbf I_{\tau_\ell}}>M)+\mathbf P_y(T-(j+1)T/N\leq\tau_k<T-jT/N)\Big),
\end{multline*}
as $n$ goes to infinity and it is dominated by
\begin{multline*}
\sum_{\ell=0}^k\frac{T}{N}\int_E\mu_i(\mathrm dx)\,b(x)\int_E\pi(x,\mathrm dy)\,\Big(\sum_{\ell=0}^k\mathbf P_y(L_{\mathbf I_{\tau_\ell}}>L)+\sum_{\ell=0}^k\mathbf P_y(M_{\mathbf I_{\tau_\ell}}>M)\\
+\mathbf P_y(T-(j+1)T/N\leq\tau_k<T-jT/N)\Big).
\end{multline*}
Finally, for all $N\geq1$ and $L,M>0$ we have
\begin{multline*}
\limsup_n\mathbf E_{x_i}\Big[\mathbf 1_{\{\tau_1<T\}}\sum_{j=0}^{N-1}\int_{jT/N}^{(j+1)T/N}\mathrm dt\,b(X_{nt})\,G'\big(t\mu_i(b)\big)\int_E\pi(X_{nt},\mathrm dy)\\
\times\big(\phi_{T-t}^{n,k}(\bar{\mathbf Z}_{\tau_1}^z,y)-\phi_{T-jT/N}^{n,k}(\bar{\mathbf Z}_{\tau_1}^z,y)\big)\Big]\\
\leq\frac{T}{N}\Big(L\,\big(\mathbf E_{x_i}\big[\lvert\bar{\mathbf Z}_{\tau_1}^z\rvert\big]+T\,M\big)\,e^{TL}+M\Big)+\int_E\mu_i(\mathrm dx)\,b(x)\int_E\pi(x,\mathrm dy)\\
\times\Big(\sum_{\ell=0}^k\mathbf P_y(L_{\mathbf I_{\tau_\ell}}>L)+\sum_{\ell=0}^k\mathbf P_y(M_{\mathbf I_{\tau_\ell}}>M)+\frac{T}{N}\mathbf P_y(\tau_k<T)\Big),
\end{multline*}
which goes to
\begin{align*}
\int_E\mu_i(\mathrm dx)\,b(x)\int_E\pi(x,\mathrm dy)\sum_{\ell=0}^k\big(\mathbf P_y(L_{\mathbf I_{\tau_\ell}}>L)+\mathbf P_y(M_{\mathbf I_{\tau_\ell}}>M)\big)
\end{align*}
when $N$ goes to infinity.
Since 
\begin{align*}
\mathbf P_y(M_{\mathbf I_{\tau_\ell^n}^n}>M)\underset{L\to\infty}{\longrightarrow}0,
\end{align*}
and 
\begin{align*}
\mathbf P_y(L_{\mathbf I_{\tau_\ell^n}^n}>L)\underset{M\to\infty}{\longrightarrow}0,
\end{align*}
we have
\begin{multline*}
\limsup_n\mathbf E_{x_i}\Big[\mathbf 1_{\{\tau_1<T\}}\sum_{j=0}^{N-1}\int_{jT/N}^{(j+1)T/N}\mathrm dt\,b(X_{nt})\,G'\big(t\mu_i(b)\big)\int_E\pi(X_{nt},\mathrm dy)\\
\times\big(\phi_{T-t}^{n,k}(\bar{\mathbf Z}_{\tau_1}^z,y)-\phi_{T-jT/N}^{n,k}(\bar{\mathbf Z}_{\tau_1}^z,y)\big)\Big]=0.
\end{multline*}
\textit{Step 5} Using Lemma 5.1. in \cite{kagan2025averaging} and the induction hypothesis \eqref{hyp:reccurence2} the following term
\begin{align*}
\mathbf E_{x_i}\Big[\mathbf 1_{\{\tau_1<T\}}\,\sum_{j=0}^{N-1}\int_{jT/N}^{(j+1)T/N}\mathrm dt\,b(X_{nt})\,G'\big(t\mu_i(b)\big)\int_E\pi(X_{nt},\mathrm dy)\,\phi_{T-jT/N}^{n,k}(\bar{\mathbf Z}_{\tau_1}^z,y)\Big]
\end{align*}
goes to 
\begin{align*}
\mathbf E_{x_i}\Big[\mathbf 1_{\{\tau_1<T\}}\,\sum_{j=0}^{N-1}\int_{jT/N}^{(j+1)T/N}\mathrm dt\,G'\big(t\mu_i(b)\big)\int_E\mu_i(\mathrm dx)\,b(x)\int_E\pi(x,\mathrm dy)\,\phi_{T-jT/N}^k(\bar{\mathbf Z}_{\tau_1}^z,y)\Big]
\end{align*}
when $n$ goes to infinity.\\
\textit{Step 6} Finally, the previous term goes to
\begin{align*}
\mathbf E_{x_i}\Big[\mathbf 1_{\{\tau_1<T\}}\,\int_0^T\mathrm dt\,G'\big(t\mu_i(b)\big)\int_E\mu_i(\mathrm dx)\,b(x)\int_E\pi(x,\mathrm dy)\,\phi_{T-t}^k(\bar{\mathbf Z}_{\tau_1}^z,y)\Big]
\end{align*}
as $N$ goes to infinity. The previous term equals
\begin{align*}
\mathbf E_{x_i}\big[\mathbf 1_{\{\tau_1<T\}}\,\phi_{T-\tau_1}^k(\bar{\mathbf Z}_{\tau_1}^z,\mathbf Y_1)\big]&=
\mathbf E_{x_i}\big[\varphi(\bar{\mathbf Z}_T^z)\,\mathbf 1_{\{\tau_1<T\leq\tau_{k+1}^n\}}\big],
\end{align*}
by the strong Markov property applied at time $\tau_1$ to the process $(\bar{\mathbf Z},\mathbf Y)$ and conclues the proof of the theorem.
\end{proof}

\section{Proof of Theorem \ref{th:convFD}}
 
For simplicity of notation, and since the computations are identical in the general case, we restrict the proof to the case $N=2$. To prove the convergence of the finite-dimensional distributions, we split the argument into two cases. 

First, we consider the case where both times lie in the same jump interval, namely when there exists $k \ge 1$ such that
\[
\tau_{k-1}^n < t_1 < t_2 \le \tau_k^n.
\]

Second, we treat the case where the two times belong to distinct jump intervals, that is, when there exist $j \ge 1$ and $k \ge j+1$ such that
\[
\tau_{j-1}^n < t_1 \le \tau_j^n
\quad \text{and} \quad
\tau_{k-1}^n < t_2 \le \tau_k^n.
\]

\subsection{Finite-dimensional convergence on a single jump interval} 

\begin{proposition}
\label{prop:convFD1}
Let $x\in E$, $z\in\mathbb R^d$ and $0<t_1<t_2$, then for all $k\geq1$ and $f_1,f_2\in C_b^1(\mathbb R^d)$ we have
\begin{align*}
\mathbf E_x\big[f_1(\mathbf Z_{t_1}^{n,z})\,f_2(\mathbf Z_{t_2}^{n,z})\,\mathbf 1_{\{\tau_{k-1}^n<t_1<t_2\leq\tau_k^n\}}\big]\underset{n\to\infty}{\longrightarrow}\mathbf E_x\big[f_1(\bar{\mathbf Z}_{t_1}^z)\,f_2(\bar{\mathbf Z}_{t_2}^z)\,\mathbf 1_{\{\tau_{k-1}<t_1<t_2\leq\tau_k\}}\big].
\end{align*}
\end{proposition}
To prove this Proposition by induction on $k\geq1$, we need the following two lemmas which are analogues of lemmas \ref{lem:majoration1} and \ref{lem:majoration2}
\begin{lemma}
For all $n,k\geq1$, $0<S<T$, $x\in E$ and $y\in\mathbb R^d$ we set
\begin{align*}
\phi_{S,T}^{n,k}(y,x)=\mathbf E_{x}\big[f_1(\mathbf Z_S^{n,y})\,f_2(\mathbf Z_T^{n,y})\mathbf 1_{\{\tau_{k-1}^n<S<T\leq\tau_k^n\}}\big].
\end{align*}
Let $L>0$, then for all $n,k\geq1$, $T>0$ there exist $M_T^k>0$ such that for all $z,y\in\mathbb R^d$ and $x\in E$
\begin{align*}
\lvert\phi_{S,T}^{n,k}(y,x)-\phi_{S,T}^{n,k}(z,x)\rvert&\leq2\,\|f_1\|_\infty\,\|f_2\|_\infty\,\big(\mathbf P_x(S\leq\tau_k^n,L_S^n>L)+\mathbf P_x(T\leq\tau_k^n,L_T^n>L)\big)\\
&+\lvert y-z\rvert\,\big(\|f_1'\|_\infty\,\|f_2\|_\infty\,e^{SL}+\|f_1\|_\infty\,\|f_2'\|_\infty\,e^{TL}\big),
\end{align*}
where for all $t>0$ and $n\geq1$
\begin{align*}
L_t^n:=\sup_{s\in[0,t]}L_{\mathbf I_s^n}.
\end{align*}
\end{lemma}
\begin{proof}
Let $L>0$, for all $n,k\geq1$, $T>0$, $x\in E$ and $y,z\in\mathbb R^d$ we have
\begin{align*}
&\lvert\phi_{S,T}^{n,k}(y,x)-\phi_{S,T}^{n,k}(z,x)\rvert\\
&=\Big\lvert\mathbf E_x\big[f_1(\mathbf Z_S^{n,y})\,f_2(\mathbf Z_T^{n,y})\,\mathbf 1_{\{\tau_{k-1}^n<S<T\leq\tau_k^n\}}\big]-\mathbf E_x\big[f_1(\mathbf Z_S^{n,z})\,f_2(\mathbf Z_T^{n,z})\,\mathbf 1_{\{\tau_{k-1}^n<S<T\leq\tau_k^n\}}\big]\Big\rvert\\
&\leq\Big\lvert\mathbf E_x\big[(f_1(\mathbf Z_S^{n,y})-f_1(\mathbf Z_S^{n,z}))\,f_2(\mathbf Z_T^{n,y})\,\mathbf 1_{\{\tau_{k-1}^n<S<T\leq\tau_k^n\}}\big]\Big\rvert\\
&+\Big\lvert\mathbf E_x\big[(f_2(\mathbf Z_T^{n,y})-f_2(\mathbf Z_T^{n,z}))\,f_1(\mathbf Z_S^{n,z})\,\mathbf 1_{\{\tau_{k-1}^n<S<T\leq\tau_k^n\}}\big]\Big\rvert\\
&\leq\|f_2\|_\infty\,\mathbf E_x\big[\lvert f_1(\mathbf Z_S^{n,y})-f_1(\mathbf Z_S^{n,z})\rvert\,\mathbf 1_{\{S\leq\tau_k^n\}}\big]\\
&+\|f_1\|_\infty\,\mathbf E_x\big[\lvert f_2(\mathbf Z_T^{n,y})-f_2(\mathbf Z_T^{n,z})\rvert\,\mathbf 1_{\{T\leq\tau_k^n\}}\big]
\end{align*}
Then, as in the proof of Lemma \ref{lem:majoration1} we have, for all $L>0$ 
\begin{align*}
\mathbf E_x\big[\lvert f_1(\mathbf Z_S^{n,y})-f_1(\mathbf Z_S^{n,z})\rvert\,\mathbf 1_{\{S\leq\tau_k^n\}}\big]\leq2\|f_1\|_\infty\,\,\mathbf P_x(S\leq\tau_k^n,L_S^n>L)+\|f_1'\|_\infty\,\lvert y-z\rvert\,e^{SL}
\end{align*}
and
\begin{align*}
\mathbf E_x\big[\lvert f_2(\mathbf Z_T^{n,y})-f_2(\mathbf Z_T^{n,z})\rvert\,\mathbf 1_{\{T\leq\tau_k^n\}}\big]\leq2\|f_2\|_\infty\,\mathbf P_x(T\leq\tau_k^n,L_T^n>L)+\|f_2'\|_\infty\,\lvert y-z\rvert\,e^{TL}.
\end{align*}
The second one is bounded above by
\begin{align*}
2\,\|f_1\|_\infty\,\|f_2\|_\infty\,\mathbf P_x(T\leq\tau_k^n,L_T^n>L).
\end{align*}
which conclues the proof of the Lemma.
\end{proof}
\begin{lemma}
For all $0<s<t$, $n,k\geq 1$, $x\in E$ and $z\in\mathbb R^d$ we set 
\begin{align*}
\phi_{s,t}^{n,k}(z,x):=\mathbf E_x\big[f_1(\mathbf Z_s^{n,z})\,f_2(\mathbf Z_t^{n,z})\,\mathbf 1_{\{\tau_{k-1}^n<s<t\leq\tau_k^n\}}\big].
\end{align*}
Then, for all $T>0$, $0<s_1<s_2<T$, $t_1<t_2<T$ and $L,M>0$ we have
\begin{align*}
&\lvert\phi_{s_1,t_1}^{n,k}(y,x)-\phi_{s_2,t_2}^{n,k}(y,x)\rvert\\
&\leq\big(\|f_1'\|_\infty\,\|f_2\|_\infty\,(s_2-s_1)+\|f_1\|_\infty\,\|f_2'\|_\infty\,(t_2-t_1)\big)\,\Big(L\,(\lvert z\rvert+TM)\,e^{TL}+M\Big)\\
&+2\,\|f_1\|_\infty\,\|f_2\|_\infty\,\big(\mathbf P_x(s_2\leq\tau_k^n,L_{s_2}^n>L)+\mathbf P_x(s_2\leq\tau_k^n,M_{s_2}^n>M)\big)\\
&+\|f_1\|_\infty\,\|f_2\|_\infty\,\mathbf P_x(s_1\leq\tau_{k-1}^n<s_2)+\mathbf P_x(t_1\leq\tau_k^n<t_2)+\mathbf P_x(t_2\leq\tau_k^n,L_{t_2}^n>L)\\
&+\mathbf P_x(t_2\leq\tau_k^n,M_{t_2}^n>M)\big)+\mathbf P_x(s_1\leq\tau_{k-1}^n<s_2)+\mathbf P_x(t_1\leq\tau_k^n<t_2)\big).
\end{align*}
\end{lemma}
\begin{proof}
Let $s_1<s_2$ and $t_1<t_2$, then, for all $x\in E$, $z\in\mathbb R^d$ and $n,k\geq1$ we have
\begin{align*}
\lvert\phi_{s_1,t_1}^{n,k}(z,x)-\phi_{s_2,t_2}^{n,k}&(z,x)\rvert\\
&=\lvert\mathbf E_{x}\big[f_1(\mathbf Z_{s_1}^{n,y})\,f_2(\mathbf Z_{t_1}^{n,y})\,\mathbf 1_{\{\tau_{k-1}^n<s_1<t_1\leq\tau_k^n\}}\big]\\
&-\mathbf E_{x}\big[f_1(\mathbf Z_{s_2}^{n,y})\,f_2(\mathbf Z_{t_2}^{n,y})\mathbf 1_{\{\tau_{k-1}^n<s_2<t_2\leq\tau_k^n\}}\big]\rvert\\
&\leq\|f_2\|_\infty\,\mathbf E_x\big[\lvert f_1(\mathbf Z_{s_1}^{n,y})\,\mathbf 1_{\{\tau_{k-1}^n<s_1<t_1\leq\tau_k^n\}}-f_1(\mathbf Z_{s_2}^{n,y})\,\mathbf 1_{\{\tau_{k-1}^n<s_2<t_2\leq\tau_k^n\}}\rvert\big]\\
&+\|f_1\|_\infty\,\mathbf E_x\big[\lvert f_2(\mathbf Z_{t_1}^{n,y})\,\mathbf 1_{\{\tau_{k-1}^n<s_1<t_1\leq\tau_k^n\}}-f_2(\mathbf Z_{t_2}^{n,y})\,\mathbf 1_{\{\tau_{k-1}^n<s_2<t_2\leq\tau_k^n\}}\rvert\big].
\end{align*}
The first term above can be rewritten as follow
\begin{align*}
\mathbf E_x\big[\lvert f_1(\mathbf Z_{s_1}^{n,y})\,\mathbf 1_{\{\tau_{k-1}^n<s_1<t_1\leq\tau_k^n\}}&-f_1(\mathbf Z_{s_2}^{n,y})\,\mathbf 1_{\{\tau_{k-1}^n<s_2<t_2\leq\tau_k^n\}}\rvert\big]\\
&=\mathbf E_x\big[\lvert f_1(\mathbf Z_{s_1}^{n,y})-f_1(\mathbf Z_{s_2}^{n,y})\rvert\,\mathbf 1_{\{\tau_{k-1}^n<s_1<t_1\leq\tau_k^n\}}\,\mathbf 1_{\{\tau_{k-1}^n<s_2<t_2\leq\tau_k^n\}}]\\
&+\mathbf E_x\big[\lvert f_1(\mathbf Z_{s_1}^{n,y})\rvert\,\mathbf 1_{\{\tau_{k-1}^n<s_1<t_1\leq\tau_k^n\}}\,\mathbf 1_{\{\tau_{k-1}^n<s_2<t_2\leq\tau_k^n\}^c}\big]\\
&+\mathbf E_x\big[\lvert f_1(\mathbf Z_{s_2}^{n,y})\rvert\,\mathbf 1_{\{\tau_{k-1}^n<s_2<t_2\leq\tau_k^n\}}\,\mathbf 1_{\{\tau_{k-1}^n<s_1<t_1\leq\tau_k^n\}^c}\big]\\
&\leq\mathbf E_x\big[\lvert f_1(\mathbf Z_{s_1}^{n,y})-f_1(\mathbf Z_{s_2}^{n,y})\rvert\,\mathbf 1_{\{s_2\leq\tau_k^n\}}\big]\\
&+\|f_1\|_\infty\,\big(\mathbf P_x(t_1\leq\tau_k^n<t_2)+\mathbf P_x(s_1\leq\tau_{k-1}^n<s_2)\big).
\end{align*}
Then, as in the proof of Lemma \ref{lem:majoration2}, for all $L,M>0$, the first term above is dominated by
\begin{gather*}
(s_2-s_1)\,\|f_1'\|_\infty\,\Big(L\,\big(\lvert z\rvert+TM\big)\,e^{TL}+M\Big)+2\,\|f_1\|_\infty\,\big(\mathbf P_x(s_2\leq\tau_k^n,L_{s_2}^n>L)\\
+\mathbf P_x(s_2\leq\tau_k^n,M_{s_2}^n>M)\big)
\end{gather*}
and we also have
\begin{align*}
&\mathbf E_x\big[\lvert f_2(\mathbf Z_{t_1}^{n,y})\,\mathbf 1_{\{\tau_{k-1}^n<s_1<t_1\leq\tau_k^n\}}-f_2(\mathbf Z_{t_2}^{n,y})\,\mathbf 1_{\{\tau_{k-1}^n<s_2<t_2\leq\tau_k^n\}}\rvert\big]\\
&\leq\mathbf E_x\big[\lvert f_2(\mathbf Z_{t_1}^{n,y})-f_2(\mathbf Z_{t_2}^{n,y})\rvert\,\mathbf 1_{\{t_2\leq\tau_k^n\}}]+\|f_2\|_\infty\,\mathbf P_x(s_1\leq\tau_{k-1}^n<s_2)+\,\|f_2\|_\infty\,\mathbf P_x(t_1\leq\tau_k^n<t_2)\\
&\leq\|f_2'\|_\infty\,(t_2-t_1)\,\Big(L\,(\lvert z\rvert+TM)\,e^{TL}+M\Big)\\
&+2\,\|f_2\|_\infty\,\big(\mathbf P_x(t_2\leq\tau_k^n,L_{t_2}^n>L)+\mathbf P_x(t_2\leq\tau_k^n,M_{t_2}^n>M)\\
&+\mathbf P_x(s_1\leq\tau_{k-1}^n<s_2)+\mathbf P_x(t_1\leq\tau_k^n<t_2)\big)
\end{align*}
which conclues the proof of the lemma.
\end{proof}
\begin{proof}[Proof of Proposition \ref{prop:convFD1}]
Let $x\in E$, $z\in\mathbb R^d$, $0<t_1<t_2$, we show by induction on $k\geq1$
\begin{align*}
\mathbf E_x\big[f_1(\mathbf Z_{t_1}^{n,z})\,f_2(\mathbf Z_{t_2}^{n,z})\,\mathbf 1_{\{\tau_{k-1}^n<t_1<t_2\leq\tau_k^n\}}\big]\underset{n\to\infty}{\longrightarrow}\mathbf E_x\big[f_1(\bar{\mathbf Z}_{t_1}^z)\,f_2(\bar{\mathbf Z}_{t_2}^z)\,\mathbf 1_{\{\tau_{k-1}<t_1<t_2\leq\tau_k\}}\big].
\end{align*}
\textit{Step 1} For $k=1$ we have
\begin{align*}
\mathbf E_x\big[f_1(\mathbf Z_{t_1}^{n,z})\,f_2(\mathbf Z_{t_2}^{n,z})\,\mathbf 1_{\{t_1<t_2\leq\tau_1^n\}}\big].
\end{align*}
Then, 
\begin{align*}
&\Big\lvert\mathbf E_x\big[f_1(\mathbf Z_{t_1}^{n,z})\,f_2(\mathbf Z_{t_2}^{n,z})\,\mathbf 1_{\{0<t_1<t_2\leq\tau_1^n\}}\big]-\mathbf E_x\big[f_1(\bar{\mathbf Z}_{t_1}^z)\,f_2(\bar{\mathbf Z}_{t_2}^z)\,\mathbf 1_{\{0<t_1<t_2\leq\tau_1\}}\big]\Big\rvert\\
&\leq\Big\lvert\mathbf E_x\big[(f_1(\mathbf Z_{t_1}^{n,z})\,\mathbf 1_{\{t_1\leq\tau_1^n\}}-f_1(\bar{\mathbf Z}_{t_1}^z)\,\mathbf 1_{\{t_1\leq\tau_1\}})\,f_2(\mathbf Z_{t_2}^{n,z})\,\mathbf 1_{\{t_2\leq\tau_1^n\}}\big]\Big\rvert\\
&+\Big\lvert\mathbf E_x\big[(f_2(\mathbf Z_{t_2}^{n,z})\,\mathbf 1_{\{t_2\leq\tau_1^n\}}-f_2(\bar{\mathbf Z}_{t_2}^z)\,\mathbf 1_{\{t_2\leq\tau_1\}})\,f_1(\bar{\mathbf Z}_{t_1}^z)\,\mathbf 1_{\{t_1\leq\tau_1\}}\big]\Big\rvert\\
&\leq\|f_2\|_\infty\,\mathbf E_x\big[\lvert f_1(\mathbf Z_{t_1}^{n,z})\,\mathbf 1_{\{t_1\leq\tau_1^n\}}-f_1(\bar{\mathbf Z}_{t_1}^z)\,\mathbf 1_{\{t_1\leq\tau_1\}}\rvert\big]\\
&+\|f_1\|_\infty\,\mathbf E_x\big[\lvert f_2(\mathbf Z_{t_2}^{n,z})\,\mathbf 1_{\{t_2\leq\tau_1^n\}}-f_2(\bar{\mathbf Z}_{t_2}^z)\,\mathbf 1_{\{t_2\leq\tau_1\}}\rvert\big].
\end{align*}
We have
\begin{align*}
\mathbf E_x\big[\lvert f_1(\mathbf Z_{t_1}^{n,z})\,\mathbf 1_{\{t_1\leq\tau_1^n\}}-f_1(&\bar{\mathbf Z}_{t_1}^z)\,\mathbf 1_{\{t_1\leq\tau_1\}}\rvert\big]\\
&=\mathbf E_x\big[\lvert f_1(\mathbf Z_{t_1}^{n,z})-f_1(\bar{\mathbf Z}_{t_1}^z)\rvert\,\mathbf 1_{\{t_1\leq\tau_1^n\wedge\tau_1\}}\big]\\
&+\mathbf E_x\big[\lvert f_1(\mathbf Z_{t_1}^{n,z})\rvert\,\mathbf 1_{\{\tau_1<t_1\leq\tau_1^n\}}\big]+\mathbf E_x\big[\lvert f_1(\bar{\mathbf Z}_{t_1}^z)\rvert\,\mathbf 1_{\{\tau_1^n<t_1\leq\tau_1\}}\big]\\
&\leq\|f_1'\|_\infty\,\mathbf E_x\big[\lvert\mathbf Z_{t_1}^{n,z}-\bar{\mathbf Z}_{t_1}^z\rvert\,\mathbf 1_{\{t_1\leq\tau_1^n\wedge\tau_1\}}\big]\\
&+\|f_1\|_\infty\,\big(\mathbf P_x(\tau_1< t_1\leq\tau_1^n)+\mathbf P_x(\tau_1^n< t_1\leq\tau_1)\big)
\end{align*}
which goes to zero when $n$ goes to infinity. Then,  we have
\begin{align*}
\mathbf E_x\big[f_1(\mathbf Z_{t_1}^{n,z})\,f_2(\mathbf Z_{t_2}^{n,z})\,\mathbf 1_{\{t_1<t_2\leq\tau_1^n\}}\big]\underset{n\to\infty}{\longrightarrow}\mathbf E_x\big[f_1(\bar{\mathbf Z}_{t_1}^z)\,f_2(\bar{\mathbf Z}_{t_2}^z)\,\mathbf 1_{\{t_1<t_2\leq\tau_1\}}\big].
\end{align*}
\textit{Step 2} For $k+1$ we have
\begin{align*}
\mathbf E_x\big[f_1(\mathbf Z_{t_1}^{n,z})\,f_2(\mathbf Z_{t_2}^{n,z})\,\mathbf 1_{\{\tau_k^n<t_1<t_2\leq\tau_{k+1}^n\}}\big].
\end{align*}
Then, for $\tau_k^n<t_1\leq\tau_{k+1}^n$ with $k\geq 1$, we have
\begin{align*}
f(\mathbf Z_{t_1}^{n,z})&=f(z)+\int_0^{t_1}f'(\mathbf Z_s^{n,z})\,F_{\mathbf I_s^n}(\mathbf Z_s^{n,z},\mathbf X_s^n)\,\mathrm ds\\
&=f(z)+\int_0^{\tau_1^n}f'(\mathbf Z_s^{n,z})\,F_{\mathbf I_s^n}(\mathbf Z_s^{n,z},\mathbf X_s^n)\,\mathrm ds+\int_{\tau_1^n}^{t_1}f'(\mathbf Z_s^{n,z})\,F_{\mathbf I_s^n}(\mathbf Z_s^{n,z},\mathbf X_s^n)\,\mathrm ds\\
&=f(\mathbf Z_{\tau_1^n}^{n,z})+\int_0^{t_1-\tau_1^n}f'(\mathbf Z_{s+\tau_1^n}^{n,z})\,F_{\mathbf I_{s+\tau_1^n}^n}(\mathbf Z_{s+\tau_1^n}^n,\mathbf X_{s+\tau_1^n}^n)\,\mathrm ds
\end{align*}
and we also have 
\begin{align*}
f(\mathbf Z_{t_2}^{n,z})&=f(\mathbf Z_{\tau_1^n}^{n,z})+\int_0^{t_2-\tau_1^n}f'(\mathbf Z_{s+\tau_1^n}^{n,z})\,F_{\mathbf I_{s+\tau_1^n}^n}(\mathbf Z_{s+\tau_1^n}^n,\mathbf X_{s+\tau_1^n}^n)\,\mathrm ds
\end{align*}
for $\tau_k^n<t_1<t_2\leq\tau_{k+1}^n$. So, by the strong Markov property applied at time $\tau_1^n$ to the process $(\mathbf Z^n,\mathbf X^n)$ we get
\begin{align*}
\mathbf E_x\big[f_1(\mathbf Z_{t_1}^{n,z})\,f_2(\mathbf Z_{t_2}^{n,z})\,\mathbf 1_{\{\tau_k^n<t_1<t_2\leq\tau_{k+1}^n\}}\big]&=\mathbf E_x\big[\mathbf 1_{\{\tau_1^n<t_1\}}\,\phi_{t_1-\tau_1^n,t_2-\tau_1^n}^{n,k}(\mathbf Z_{\tau_1^n}^{n,z},\mathbf X_{\tau_1^n}^n)\big]
\end{align*}
where for all $0<s<t$, $x\in E$ and $z\in\mathbb R^d$
\begin{align*}
\phi_{s,t}^{n,k}(z,x):=\mathbf E_x\big[f_1(\mathbf Z_s^{n,z})\,f_2(\mathbf Z_t^{n,z})\,\mathbf 1_{\{\tau_{k-1}^n<s<t\leq\tau_k^n\}}\big],\quad\forall n,k\geq1.
\end{align*}
By the induction hypothesis, for all $0<s<t$, $y\in E$ and $w\in\mathbb R^d$
\begin{align*}
\phi_{s,t}^{n,k}(w,y)\underset{n\to\infty}{\longrightarrow}\phi_{s,t}^k(w,y):=\mathbf E_x\big[f_1(\bar{\mathbf Z}_s^w)\,f_2(\bar{\mathbf Z}_t^w)\,\mathbf 1_{\{\tau_{k-1}^n<s<t\leq\tau_k^n\}}\big]\quad\forall k\geq1.
\end{align*}
Using previous lemmas and the same methods as in steps 2, 3, 4 and 5 of the proof of Theorem \ref{th:conv}, we show that 
\begin{align*}
\mathbf E_x\big[\mathbf 1_{\{\tau_1^n<t_1\}}\,\phi_{t_1-\tau_1^n,t_2-\tau_1^n}^{n,k}(\mathbf Z_{\tau_1^n}^{n,z},\mathbf X_{\tau_1^n}^n)\big]\underset{n\to\infty}{\longrightarrow} \mathbf E_x\big[\mathbf 1_{\{\tau_1<t_1\}}\,\phi_{t_1-\tau_1,t_2-\tau_1}^k(\mathbf Z_{\tau_1}^z,\mathbf Y_1)\big],
\end{align*}
Then, by applying the strong Markov at time $\tau_1$ to the process $(\mathbf Z,\mathbf Y)$ we get
\begin{align*}
\mathbf E_x\big[f_1(\mathbf Z_{t_1}^{n,z})\,f_2(\mathbf Z_{t_2}^{n,z})\,\mathbf 1_{\{\tau_k^n<t_1<t_2\leq\tau_{k+1}^n\}}\big]\underset{n\to\infty}{\longrightarrow}\mathbf E_x\big[f_1(\mathbf Z_{t_1}^z)\,f_2(\mathbf Z_{t_2}^z)\,\mathbf 1_{\{\tau_k<t_1<t_2\leq\tau_{k+1}\}}\big].
\end{align*}
which conclues the proof of Proposition \ref{prop:convFD1}.
\end{proof}

\subsection{Finite-dimensional convergence across different jump intervals}

\begin{proposition}
\label{prop:convFD2}
Let $0<t_1<t_2$, $x\in E$, $z\in\mathbb R^d$, then for all $j\geq 1$ and $f_1,f_2\in C_b^1(\mathbb R^d)$ we have
\begin{multline*}
\mathbf E_x\big[f_1(\mathbf Z_{t_1}^{n,z})\,f_2(\mathbf Z_{t_2}^{n,z})\,\mathbf 1_{\{\tau_{j-1}^n<t_1\leq\tau_j^n,\tau_{k-1}^n<t_2\leq\tau_k^n\}}\big]\\
\underset{n\to\infty}{\longrightarrow}\mathbf E_x\big[f_1(\bar{\mathbf Z}_{t_1}^z)\,f_2(\bar{\mathbf Z}_{t_2}^z)\,\mathbf 1_{\{\tau_{j-1}<t_1\leq\tau_j,\tau_{k-1}<t_2\leq\tau_k\}}\big],\quad\forall k\geq j+1.
\end{multline*}
\end{proposition}
To prove this Proposition, we need the two following lemmas
\begin{lemma}
Let $S,T>0$, $n,j\geq1$, $k\geq j+1$, we set for all $f_1,f_2\in C_b^1(\mathbb R^d)$, $x\in E$ and $z\in\mathbb R^d$
\begin{align*}
\phi_{S,T}^{n,j,k}:=\mathbf E_x\big[f_1(\mathbf Z_S^{n,z})\,f_2(\mathbf Z_T^{n,z})\,\mathbf 1_{\{\tau_{j-1}^n<S\leq\tau_j^n,\tau_{k-1}^n<T\leq\tau_k^n\}}\big].
\end{align*}
Then, for all $L,M>0$
\begin{align*}
\lvert\phi_{S,T}^{n,j,k}(y,x)-\phi_{S,T}^{n,j,k}(z,x)\rvert&\leq\lvert y-z\rvert\,\big(\|f_1'\|_\infty\,\|f_2\|_\infty\,e^{SL}+\|f_1\|_\infty\,\|f_2'\|_\infty\,e^{TL}\big)\\
&+2\,\|f_1\|_\infty\,\|f_2\|_\infty\,\big(\mathbf P_x(s\leq\tau_j^n,L_s^n>L)+\mathbf P_x(t\leq\tau_k^n,L_t^n>L)\big).
\end{align*}
where for all $t>0$ and $n\geq1$
\begin{align*}
L_t^n:=\sup_{s\in[0,t]}L_{\mathbf I_s^n}.
\end{align*}
\end{lemma}
\begin{proof}
Let $S,T>0$, $n,j\geq1$, $f_1,f_2\in C_b^1(\mathbb R^d)$. Then, for all $x\in E$ and $z\in\mathbb R^d$
\begin{align*}
\lvert\phi_{S,T}^{n,j,k}(y,x)&-\phi_{S,T}^{n,j,k}(z,x)\rvert\\
&\leq\mathbf E_x\big[\lvert f_1(\mathbf Z_S^{n,y})\,f_2(\mathbf Z_T^{n,y})-f_1(\mathbf Z_S^{n,z})\,f_2(\mathbf Z_T^{n,z})\rvert\,\mathbf 1_{\{\tau_j^n<S\leq\tau_{j+1}^n,\tau_k^n<T\leq\tau_{k+1}^n\}}\big]\\
&\leq\|f_2\|_\infty\,\mathbf E_x\big[\lvert f_1(\mathbf Z_S^{n,y})-f_1(\mathbf Z_S^{n,z})\rvert\,\mathbf 1_{\{S\leq\tau_j^n\}}\big]\\
&+\|f_1\|_\infty\,\mathbf E_x\big[\lvert f_2(\mathbf Z_T^{n,y})-f_2(\mathbf Z_T^{n,z})\rvert\,\mathbf 1_{\{T\leq\tau_k^n\}}\big].
\end{align*}
Then, as in proof of Lemma \ref{lem:majoration1} we have 
\begin{align*}
\mathbf E_x\big[\lvert f_1(\mathbf Z_S^{n,y})-f_1(\mathbf Z_S^{n,z})\rvert\,\mathbf 1_{\{S\leq\tau_j^n\}}\big]\leq 2\|f_1\|_\infty\,\mathbf P_x(S\leq\tau_j^n, L_S^n>L)+\|f_1'\|_\infty\,\lvert y-z\rvert\,e^{SL},
\end{align*}
and 
\begin{align*}
\mathbf E_x\big[\lvert f_2(\mathbf Z_T^{n,y})-f_2(\mathbf Z_T^{n,z})\rvert\,\mathbf 1_{\{T\leq\tau_k^n\}}\big]\leq 2\,\|f_2\|_\infty\,\mathbf P_x(T\leq\tau_k^n,L_T^n>L)+\|f_2'\|_\infty\,\lvert y-z\rvert\,e^{TL},
\end{align*}
which conclues the proof of the Lemma.
\end{proof}
\begin{lemma}
Let $0<s_1<s_2$, $0<t_1<t_2$, $x\in E$, $z\in\mathbb R^d$, then, for all $L,M>0$, $n,j\geq1$ and $j\geq k+1$ we have
\begin{align*}
\big\lvert\phi_{s_1,t_1}^{n,k}(z,x)-\phi_{s_2,t_2}^{n,k}(z,x)\big\rvert&\leq\|f_1'\|_\infty\,\|f_2\|_\infty\,(s_2-s_1)\,\Big(L\,\big(\lvert z\rvert+TM)\,e^{TL}+M\Big)\\
&+2\,\|f_1\|_\infty\,\|f_2\|_\infty\,\big(\mathbf P_x(s_2\leq\tau_j^n,L_{s_2}^n>L)+\mathbf P_x(s_2\leq\tau_j^n,M_{s_2}^n>M)\big)\\
&+\|f_1\|_\infty\,\|f_2\|_\infty\,\big(\mathbf P_x(\tau_{j-1}^n<s_1\leq\tau_j^n)+\mathbf P_x(\tau_{j-1}^n<s_2\leq\tau_j^n)\big)\\
&+\|f_1\|_\infty\,\|f_2\|_\infty\,(t_2-t_1)\,\Big(L\,\big(\lvert z\rvert+TM)\,e^{TL}+M\Big)\\
&+2\,\|f_1\|_\infty\,\|f_2'\|_\infty\,\big(\mathbf P_x(t_2\leq\tau_k^n,L_{t_2}^n>L)+\mathbf P_x(t_2\leq\tau_k^n,M_{t_2}^n>M)\big)\\
&+\|f_1\|_\infty\,\|f_2\|_\infty\,\big(\mathbf P_x(\tau_{k-1}^n<t_1\leq\tau_k^n)+\mathbf P_x\big(\tau_{k-1}^n<t_2\leq\tau_k^n)\big)
\end{align*}
where for all $t>0$ and $n\geq1$
\begin{align*}
M_t^n:=\sup_{s\in[0,t]}M_{\mathbf I_s^n}.
\end{align*}
\end{lemma}
\begin{proof}
Let $T>0$, $0<s_1<s_2<T$, $0<t_1<t_2<T$. For all $x\in E$, $z\in\mathbb R^d$ we have 
\begin{align*}
\big\lvert\phi_{s_1,t_1}^{n,j,k}(z,x)&-\phi_{s_2,t_2}^{n,j,k}(z,x)\big\rvert\\
&\leq\,\|f_2\|_\infty\,\mathbf E_x\big[\lvert f_1(\mathbf Z_{s_1}^{n,z})-f_1(\mathbf Z_{s_2}^{n,z})\rvert\,\mathbf 1_{\{s_2\leq\tau_j^n\}}\big]\\
&+\|f_1\|_\infty\,\|f_2\|_\infty\,\big(\mathbf P_x(\tau_{j-1}^n<s_1\leq\tau_j^n<s_2)+\mathbf P_x(s_1<\tau_{j-1}^n<s_2\leq\tau_j^n)\big)\\
&+\|f_1\|_\infty\,\mathbf E_x\big[\lvert f_2(\mathbf Z_{t_1}^{n,z})-f_2(\mathbf Z_{t_2}^{n,z})\rvert\,\mathbf 1_{\{t_2\leq\tau_k^n\}}\big]\\
&+\|f_1\|_\infty\,\|f_2\|_\infty\,\big(\mathbf P_x(\tau_{k-1}^n<t_1\leq\tau_k^n<t_2)+\mathbf P_x\big(t_1<\tau_{k-1}^n<t_2\leq\tau_k^n)\big).
\end{align*}
The result is obtained by using the same inequalities as in the proof of Lemma \ref{lem:majoration2}. 
\end{proof}

\begin{proof}[Proof of Propostion \ref{prop:convFD2}]
We also make the proof by induction on $j\geq1$. 

\medskip

\textit{Step 1} For $j=1$ and $k\geq j+1$ we have, for all $x\in E$, $z\in\mathbb R^d$ and $f_1,f_2\in C_b^1(\mathbb R^d)$
\begin{align*}
\mathbf E_x\big[f_1(\mathbf Z_{t_1}^{n,z})\,f_2(\mathbf Z_{t_2}^{n,z})\,\mathbf 1_{\{t_1\leq\tau_1^n,\tau_{k-1}^n<t_2\leq\tau_k^n\}}\big].
\end{align*}
So, by the strong Markov property applied to $(\mathbf Z^n,\mathbf X^n)$, at time $\tau_1^n$, the previous term is equal to
\begin{align*}
\mathbf E_x\big[f_1(\mathbf Z_{t_1}^{n,z})\,\mathbf 1_{\{t_1\leq\tau_1^n<t_2\}}\,\phi_{t_2-\tau_1^n}^{n,k}(\mathbf Z_{\tau_1^n}^{n,z},\mathbf X_{\tau_1^n}^n)\big].
\end{align*}
By addition and substraction, the previous term equals
\begin{multline*}
\mathbf E_x\big[\big(f_1(\mathbf Z_{t_1}^{i,n,z})-f_1(\bar{\mathbf Z}_{t_1}^{i,z})\big)\,\mathbf 1_{\{t_1\leq\tau_1^n<t_2\}}\,\phi_{t_2-\tau_1^n}^{n,k}(\mathbf Z_{\tau_1^n}^{n,z},\mathbf X_{\tau_1^n}^n)\big]\\
+\mathbf E_x\big[f_1(\bar{\mathbf Z}_{t_1}^{i,z})\,\mathbf 1_{\{t_1\leq\tau_1^n<t_2\}}\,\phi_{t_2-\tau_1^n}^{n,k}(\mathbf Z_{\tau_1^n}^{n,z},\mathbf X_{\tau_1^n}^n)\big].
\end{multline*}
Since $f_2$ is bounded, the first term above goes to zero when $n$ goes to infinity by \ref{theo:convSup}. Since $\|f_1\|_\infty<\infty$, using the same calculations as in the proof of Theorem \ref{th:conv} the second term above goes to   
\begin{align*}
\mathbf E_x\big[f_1(\bar{\mathbf Z}_{t_1}^{i,z})\,\mathbf 1_{\{t_1\leq\tau_1^n<t_2\}}\,\phi_{t_2-\tau_1}^k(\bar{\mathbf Z}_{\tau_1}^z,\mathbf Y_1)\big]
\end{align*}
when $n$ goes to infinity. By the strong Markov property of the process $(\bar{\mathbf Z},\mathbf Y)$ at time $\tau_1$, the previous term is equal to
\begin{align*}
\mathbf E_x\big[f_1(\bar{\mathbf Z}_{t_1}^z)\,f_2(\mathbf Z_{t_2}^z)\,\mathbf 1_{\{t_1\leq\tau_1^n,\tau_{k-1}^n<t_2\leq\tau_k^n\}}\big].
\end{align*}
Thus, the property is satisfied for $j=1$.

\medskip

\textit{Step 2} For $j+1$ and $k\geq j+2$, by the strong Markov property applied at time $\tau_1^n$ to the process $(\mathbf Z^n,\mathbf X^n)$ we have
\begin{align*}
\mathbf E_x\big[f_1(\mathbf Z_{t_1}^{n,z})\,f_2(\mathbf Z_{t_2}^{n,z})\,\mathbf 1_{\{\tau_j^n<t_1\leq\tau_{j+1}^n,\tau_{k-1}^n<t_2\leq\tau_k^n\}}\big]&=\mathbf E_x\big[\mathbf 1_{\{\tau_1^n<t_1\}}\,\phi_{t_1-\tau_1^n,t_2-\tau_1^n}^{n,j,k}(\mathbf Z_{\tau_1^n}^{n,z},\mathbf X_{\tau_1^n}^n)\big]
\end{align*}
where, for all $s,t>0$, $x\in E$, $z\in\mathbb R^d$, $n,j\geq1$ and $k\geq j+2$ 
\begin{align*}
\phi_{s,t}^{n,j,k-1}(z,x):=\mathbf E_x\big[f_1(\mathbf Z_s^{n,z})\,f_2(\mathbf Z_t^{n,z})\,\mathbf 1_{\{\tau_{j-1}^n<s\leq\tau_j^n,\tau_{k-2}^n<t\leq\tau_{k-1}^n\}}\big].
\end{align*}
By the induction property we have 
\begin{align*}
\phi_{s,t}^{n,j,k-1}(z,x)\underset{n\to\infty}{\longrightarrow}\phi_{s,t}^{j,k-1}(z,x):=\mathbf E_x\big[f_1(\bar{\mathbf Z}_s^z)\,f_2(\bar{\mathbf Z}_t^z)\,\mathbf 1_{\{\tau_{j-1}<s\leq\tau_j^n,\tau_{k-2}<t\leq\tau_{k-1}\}}\big].
\end{align*}
Using the two previous Lemmas, by the same computations as in the step 2, 3, 4 and  of the proof of Theorem \ref{th:conv} we have
\begin{align*}
\mathbf E_x\big[\mathbf 1_{\{\tau_1^n<t_1\}}\,\phi_{t_1-\tau_1^n,t_2-\tau_1^n}^{n,j,k}(\mathbf Z_{\tau_1^n}^{n,z},\mathbf X_{\tau_1^n}^n)\big]\underset{n\to\infty}{\longrightarrow}\mathbf E_x\big[\mathbf 1_{\{\tau_1<t_1\}}\,\phi_{t_1-\tau_1,t_2-\tau_1}^{j,k}(\mathbf Z_{\tau_1}^{n,z},\mathbf Y_1)\big]
\end{align*}
which is equal to
\begin{align*}
\mathbf E_x\big[f_1(\bar{\mathbf Z}_{t_1}^z)\,f_2(\bar{\mathbf Z}_{t_2}^z)\,\mathbf 1_{\{\tau_j<t_1\leq\tau_{j+1},\tau_{k-1}<t_2\leq\tau_k\}}\big]
\end{align*}
and conclues the proof of Proposition \ref{prop:convFD2}.
\end{proof}

\subsection{Proof of Theorem 5: convergence of the finite-dimensional distributions of \texorpdfstring{$\mathbf Z^n$}{Zn}}

\label{subsec:proofConvFD}

\begin{proof}[Proof of Theorem \ref{th:convFD}]
We prove this for the case where $N=2$, but the result holds in the same way for any $n\geq2$. Let $0<t_1<t_2$, we have
\begin{align*}
\mathbf E_x\big[f_1(\mathbf Z_{t_1}^{n,z})\,&f_2(\mathbf Z_{t_2}^{n,z})\,\mathbf 1_{\{t_2\leq\tau_k^n\}}\big]\\
&=\sum_{k_1=0}^{k-1}\sum_{k_2=k_1}^{k-1}\mathbf E_x\big[f_1(\mathbf Z_{t_1}^{n,z})\,\mathbf 1_{\{\tau_{k_1}^n<t_1\leq\tau_{k_1+1}^n\}}\,f_2(\mathbf Z_{t_2}^{n,z})\,\mathbf 1_{\{\tau_{k_2}^n<t_2\leq\tau_{k_2+1}^n\}}\big]\\
&=\sum_{k_1=0}^{k-1}\mathbf E_x\big[f_1(\mathbf Z_{t_1}^{n,z})\,f_2(\mathbf Z_{t_2}^{n,z})\,\mathbf 1_{\{\tau_{k_1}^n<t_1<t_2\leq\tau_{k_1+1}^n\}}\big]\\
&+\sum_{k_1=0}^{k-1}\sum_{k_2=k_1+1}^{k-1}\mathbf E_x\big[f_1(\mathbf Z_{t_1}^{n,z})\,f_2(\mathbf Z_{t_2}^{n,z})\,\mathbf 1_{\{\tau_{k_1}^n<t_1\leq\tau_{k_1+1}^n,\tau_{k_2}^n<t_2\leq\tau_{k_2+1}^n\}}\big].
\end{align*}
Then, by Propositions \ref{prop:convFD1} and \ref{prop:convFD2} we have
\begin{align*}
\lim_{n\to\infty}\mathbf E_x\big[f_1(\mathbf Z_{t_1}^{n,z})\,&f_2(\mathbf Z_{t_2}^{n,z})\,\mathbf 1_{\{t_2\leq\tau_k^n\}}\big]\\
&=\sum_{k_1=0}^{k-1}\mathbf E_x\big[f_1(\bar{\mathbf Z}_{t_1}^z)\,f_2(\bar{\mathbf Z}_{t_2}^z)\,\mathbf 1_{\{\tau_{k_1}<t_1<t_2\leq\tau_{k_1+1}\}}\big]\\
&+\sum_{k_1=0}^{k-1}\sum_{k_2=k_1+1}^{k-1}\mathbf E_x\big[f_1(\bar{\mathbf Z}_{t_1}^z)\,f_2(\bar{\mathbf Z}_{t_2}^z)\,\mathbf 1_{\{\tau_{k_1}<t_1\leq\tau_{k_1+1},\tau_{k_2}<t_2\leq\tau_{k_2+1}\}}\big], 
\end{align*}
which is equal to
\begin{align*}
\mathbf E_x\big[f_1(\bar{\mathbf Z}_{t_1}^z)\,f_2(\bar{\mathbf Z}_{t_2}^z)\,\mathbf 1_{\{t_2\leq\tau_k\}}\big]
\end{align*}
and conclues the proof of the Theorem.
\end{proof}

\section{Examples}

In this section we consider several examples. The first example we look at is inspired by \cite{faggionato2008averaging} and \cite{Pakdaman2012} which fits our framework and satisfies our hypothesis given by \eqref{hyp:Lip} and \eqref{hyp:Fbound}.

\subsection{Multiscale piecewise-deterministic Markov processes}

In contrast to the example studied in \cite{faggionato2008averaging} and \cite{Pakdaman2012}, we consider the case where the state space $E$ is infinite and the holding times are no longer exponentially distributed, but instead have a general distribution with cumulative distribution function $G$. Furthermore, our model construction makes it unnecessary to describe the behavior of the process restricted to each $E_i$, $i \geq 1$, and thus to specify either the jump rates or the distribution of post-jump positions within each $E_i$. In particular, there is no need to assume that the process is Markovian. As a result, the limiting averaged stochastic differential equation is no longer driven by a pure-jump Markov process, but rather by a pure-jump semi-Markov process, which extend the generality of the model.

\medskip

\textbf{Two-time scale PDMP $(\mathbf Z^n,\mathbf X^n)$.} \textit{Underlying ergodic process $X$.} We consider a stochastic jump process $X$ taking values in a countable set $E$, which can be decomposed as a disjoint union
\[
E = \bigsqcup_{i \geq 1} E_i,
\]
where $(E_i)_{i \geq 1}$ is a partition of $E$ into finite subsets such that for all $i\geq1$ and all $x_i\in E_i$,
\begin{align*}
\mathbf P_{x_i}\big(X_t\in E_i,\ \forall t\geq0\big)=1.
\end{align*} For each $i \geq 1$, we assume that the process $X$ restricted to $E_i$ admits a unique invariant probability measure $(\mu_i(x))_{x \in E_i}$.

\medskip

\textbf{Fast process $\mathbf X^n$.} We construct a process $\mathbf X$ by a piecing-out procedure based on $X$, whose first jump time corresponds to the exit time from a block $E_i$, with a distribution specified below.

\medskip

Let $G:\mathbb{R}_+ \to [0,1]$ be an absolutely continuous cumulative distribution function, and let $\lambda:E \to \mathbb{R}_+$ be a measurable function. For $x \in E$, the holding time $\sigma_1$ at $x$ satisfies
\[
\mathbb{P}_x(\sigma_1 \leq t)
=
G\Big(\int_0^t \lambda(X_s)\,\mathrm ds\Big).
\]
To define the transition at jump-time we considere a probability kernel $\pi:E\times E\to[0,1]$ and for all $i\geq1$, starting from a state $x\in E_i$, we have
\begin{align*}
\mathbf P_x(\mathbf X_{\tau_1}=z\,\lvert\,\mathbf X_{\tau_1-}=y)&=\pi(y,z),\quad\forall z\in E_i\quad\textup{and}\quad y\not\in E_i 
\end{align*} 
with 
\begin{align*}
\sum_{w\not\in E_i}\pi(y,w)=1.
\end{align*}
For all $n\geq1$ we consider the accelerated process $\mathbf X^n$ where we take the accelerated process $X^n$ defined for all $n\geq0$ by
\begin{align*}
X_t^n=X_{nt},\quad\forall t\geq0,
\end{align*}
instead of $X$.

\medskip

\textbf{Index process $\mathbf I^n$.} We remind the definition of the index process $\mathbf I^n$:
\[\mathbf I_t^n=i\iff \mathbf X_t^n\in E_i,\quad\forall t\geq0,
\]
and we denote by $(\tau_k^n)$ its sequence of jump-times.

\medskip

\textbf{Multiscale stochastic differential equation and PDMP $(\mathbf Z^n,\mathbf X^n)$.} We now define a two-time-scale PDMP $(\mathbf Z^n,\mathbf X^n)$ on $\mathbb R^d \times E$. The fast component is $\mathbf X^n$, while the slow component $\mathbf Z^n$ evolves continuously according to the stochastic differential equation
\[
\frac{\mathrm d \mathbf Z_t^n}{\mathrm dt}
=
f(\mathbf Z_t^n,\mathbf X_t^n),
\qquad
\mathbf Z_0^n = z \in \mathbb R^d,
\]
where $f:\mathbb R^d \times E \to \mathbb R^d$ is a measurable function. \textit{Desciption of the two-time scal process $(\mathbf Z^n,\mathbf X^n)$.} The sample paths of the PDMP $(\mathbf Z^n,\mathbf X^n)$ are constructed iteratively as follows. One starts at time $\tau_0^n=0$ from an initial condition $(z_0,x_0)\in\mathbb R^d\times E_{i_0}$, with $i_0\geq1$. Between two successive jump times $\tau_0^n$ and $\tau_1^n$, i.e. before the first transition of the discrete component $\mathbf X^n$, the process $\mathbf Z^n$ evolves deterministically according to the ordinary differential equation
\[
\frac{\mathrm d \mathbf Z_t^n}{\mathrm dt}=f(\mathbf Z_t^n,x_0), 
\qquad \mathbf Z_{\tau_0^n}^n=z_0,
\]
while the fast component evolves as $\mathbf X_t^n=X_{nt}$. At time $\tau_1^n$, the process $\mathbf X^n$ jumps from $x_0$ to a new state $x_1\in E$, with $x_1\notin E_{i_0}$, according to the transition mechanism described above. The slow component is then restarted from the value reached at the jump time, i.e.
\[
z_1 := \mathbf Z_{\tau_1^n}^n,
\]
and the evolution continues on $[\tau_1^n,\tau_2^n)$ by solving
\[
\frac{\mathrm d \mathbf Z_t^n}{\mathrm dt}=f(\mathbf Z_t^n,x_1),
\qquad \mathbf Z_{\tau_1^n}^n=z_1.
\]

This construction is then iterated over all jump times $(\tau_k^n)_{k\geq0}$ to obtain the full càdlàg trajectory of the PDMP $(\mathbf Z^n,\mathbf X^n)$. We assume that there exists a family of vector fields $(f_i)_{i \geq 1}$ such that for $x \in E_i$,
\[
f(\cdot,x)=f_i(\,\cdot\,,x).
\]
\textbf{Averaged limiting process $\bar{\mathbf Z}$.} \textit{Averaged index process.} Let $\mathbf I$ be a semi Markov-jump process on $\mathbb N^*$. Its first jump time satisfies
\begin{align*}
\mathbf P_i(\tau_1\leq t)=G(t\,\bar\lambda_i)
\end{align*}
where
\begin{align*}
\bar\lambda_i := \sum_{y\in E_i} \mu_i(y)\, \lambda(y).
\end{align*}
and for $i,j\in\mathbb N^*$ its transition are given by
\begin{align*}
\mathbf P_i(\mathbf I_{\tau_1}=j)&=\bar\pi_{ij}\quad j\not=i,
\end{align*}
where
\begin{align*}
\bar\pi_{ij}=\sum_{x\in E_i}\sum_{y\in E_j}\,\mu_i(x)\,\pi(x,y).
\end{align*}

\medskip 

\textit{Limiting averaged stochastic differential equation $\bar{\mathbf Z}$.} Now, we define the averaged vector field, for all $z\in\mathbb R^d$
\begin{align*}
\bar f_i(z):=\sum_{x\in E_i}f_i(z,x)\,\mu_i(x),
\end{align*}
and the limiting averaged dynamics
\begin{align*}
\frac{\mathrm d\bar{\mathbf Z}_t^z}{\mathrm dt}=\bar f_{\mathbf I_t}(\bar{\mathbf Z}_t^z),\quad \bar{\mathbf Z}_0^z=z\in\mathbb R^d.
\end{align*}

Then, under the regularity conditions: for all $y,z\in\mathbb R^d$ and $x\in E$ 
\begin{align*}
\lvert f_i(y,x)-f_i(z,x)\rvert\leq L_i\lvert y-z\rvert\quad\forall i\geq1.
\end{align*} 
we have the following result.
\begin{theorem}
Let $z \in \mathbb R^d$ and $T>0$. Then the sequence of processes $(\mathbf{Z}^{n,z})_{n\geq 1}$ converges in law for the Skorokhod topology on $D([0,T],\mathbb R^d)$ to $\bar{\mathbf Z}^z$ defined above, where $D([0,T],\mathbb R^d)$ is equipped with the Skorokhod topology induced by the supremum norm.
\end{theorem}

\medskip

\begin{proof}
The result is a direct application of Theorem~\ref{thm:convSkorokhod}. We verify that its assumptions are satisfied for the PDMP $(\mathbf Z^n,\mathbf X^n)$ constructed above.

\medskip

\textit{Lipschitz condition.} For all $y,z\in\mathbb R^d$ and $x\in E$ 
\begin{align*}
\lvert f_i(y,x)-f_i(z,x)\rvert\leq L_i\lvert y-z\rvert\quad\forall i\geq1,
\end{align*} 
assumption~\eqref{hyp:Lip} is verified.

\medskip
 
\textit{Growth control.} Since $E_i$ is finite for all $i\geq1$ 
\begin{align*}
\sup_{x\in E_i}\lvert f_i(0,x)\rvert<\infty,
\end{align*}
so hypothesis \eqref{hyp:Fbound} is satisfied.

\medskip

\textit{Ergodic averaging on blocks.}
Since each $E_i$ is finite and invariant, the restriction of $X$ to $E_i$ is an irreducible finite-state Markov jump process with unique invariant measure $\mu_i$. Hence, for all $x\in E_i$,
\[
\frac{1}{t}\int_0^t \lambda(X_s)\,ds \xrightarrow[t\to\infty]{} \bar\lambda_i,
\]
with
\[
\bar\lambda_i=\sum_{y\in E_i}\mu_i(y)\lambda(y).
\]
After time rescaling $X_t^n=X_{nt}$, this yields
\[
\int_0^t \lambda(X_{ns})\,ds \xrightarrow[n\to\infty]{} t\,\bar\lambda_i,
\]
which identifies the limiting holding-time distribution $G(t\bar\lambda_i)$.

\medskip

\textit{Effective jump mechanism.}
When $\mathbf X^n$ leaves a block $E_i$, transitions are governed by the kernel $\pi$. Averaging over the invariant distribution inside each block yields the effective probabilities
\[
\bar\pi_{ij}=\sum_{x\in E_i}\sum_{y\in E_j}\mu_i(x)\pi(x,y).
\]

\medskip

\textit{Averaging of the slow dynamics.}
Between jump times, $\mathbf Z^n$ evolves according to
\[
\frac{d\mathbf Z_t^n}{dt}=f(\mathbf Z_t^n,\mathbf X_t^n), \qquad f(\cdot,x)=f_i(\cdot,x).
\]
Since $\mathbf X^n$ mixes rapidly within each $E_i$, the slow component sees the averaged vector field
\[
\bar f_i(z)=\sum_{x\in E_i} f_i(z,x)\mu_i(x).
\]
The assumptions of Theorem~\ref{th:conv} are satisfied, which yields the stated convergence.
\end{proof}

\subsection{Branching Population Models Coupled with a Fast Ergodic Environment}

In those examples, we refer to Section 4.1 in \cite{kagan2025averaging} for the construction of the branching process $\mathbf X^n$ and the corresponding assumptions. We provide only a brief description here, in order to focus primarily on the process $\mathbf Z^n$.\\

We illustrate our framework through a branching-type model, which provides a natural and insightful example of the general setting. We consider, for each $n \geq 1$, a process $(\mathbf Z_t^n)_{t \geq 0}$ taking values in a state space $Z$, defined as the solution of the ordinary differential equation
\begin{align*}
\frac{\mathrm d\mathbf Z_t^n}{\mathrm dt} = F_{\mathbf I_t^n}(\mathbf Z_t^n,\mathbf X_t^n), \quad \mathbf Z_0^n = z,
\end{align*}
where the dynamics are driven by an underlying population process.
More precisely, $(\mathbf X_t^n)_{t \geq 0}$ is a branching process describing the evolution of a population of individuals, each individual being characterized by a type taking values in a set
\begin{align*}
E:=\bigsqcup_{i\geq1}E_i
\end{align*}
where $E_i=D^i$ for all $i\geq1$ denotes the space of possible types. At any time $t \geq 0$, the population is composed of $\mathbf I_t^n$ individuals, and the configuration of their types is encoded by the vector
\begin{align*}
\mathbf X_t^n=(x_1,\ldots,x_{\mathbf I_t^n})\in E_{\mathbf I_t^n},
\end{align*}
where, for each $i \in \{1,\ldots,\mathbf I_t^n\}$, the variables $x_i$ are independent and identically distributed copies of a Markov process $Y$ taking values in $D$. In particular, the process $(\mathbf I_t^n)_{t \geq 0}$ represents the size of the population at time $t$, while $(\mathbf X_t^n)_{t \geq 0}$ captures its internal composition.
The individual dynamic $(Y_t)_{t \geq 0}$ is a Markov process on a measurable space $D$. We assume that $Y$ admits a unique ergodic stationary distribution $\chi$ such that for all $y \in D$, 
\[
\lim_{t \to \infty} \| \mathbf{P}_y( Y_t \in \cdot) - \chi \|_{TV} = 0.
\]
The individual branching rate $\beta : D \to \mathbb{R}_+$: a measurable function such that   for all $y \in D$, the family $(\frac{1}{n}\int_0^n \beta(Y_s) \, \mathrm ds)_{n \geq 1}$ is uniformly integrable. The individual reproduction is given by a transition kernel $K$ from $D$ to $\{ \Delta\}\cup \cup_{i\geq 1} D^i$, where $\Delta$ is a cemetery point corresponding to an empty progeny. We have shown in \cite{kagan2025averaging}, Proposition 4.1. that the sequence of processes $(\mathbf{I}^n)_{n\geq 1}$ converges in law for the Skorokhod topology on $D([0,+\infty),\mathbb N)$ to 
a standard continuous-time Galton-Watson tree with reproduction/branching rate $\chi(\beta)$ and offspring distribution \[L = \bigg(\frac{\chi(\beta K( \cdot, D^j))}{\chi(\beta)}\bigg)_{j \in \mathbb N}.\] which we shall denote by $\mathbf I$ in the remainder.

\subsubsection{Additive population–environment interaction}

In the additive population–environment interaction example, th process $(\mathbf Z_t^n)_{t \geq 0}$ models the evolution of an external quantity (for instance, a resource, an environmental variable, or a macroscopic observable) which interacts with the population. Its dynamics depend both on its current state $\mathbf Z_t^n$ and on the instantaneous configuration of the population $(\mathbf X_t^n, \mathbf I_t^n)$.
To make this interaction explicit, we introduce a function $\psi_1 : E \to \mathbb{R}$, such that $\|\psi_1\|_\infty<\infty$ which represents the contribution (e.g., consumption or production) of a single individual of type $x \in E$. In addition, we consider a function $g : Z \to \mathbb{R}$ describing the intrinsic dynamics of the variable $\mathbf Z^n$. We assume that $g$ is Lipschitz continuous, i.e., there exists a constant $L_g > 0$ such that
\begin{align*}
\lvert g(y) - g(z) \rvert \leq L_g \lvert y - z \rvert, \quad \forall y, z \in\mathbb R^d,
\end{align*}
which ensures well-posedness of the differential equation.
In this branching setting, the interaction term is defined, for any configuration $(x_1, \dots, x_i) \in E^i$, by
\begin{align*}
F_i^1\big(z,(x_1,\ldots,x_i)\big) = \sum_{k=1}^i \psi_1(x_k) + g(z).
\end{align*}
This expression reflects the fact that the evolution of $\mathbf Z_t^n$ results from the aggregation of the individual contributions of all members of the population, through the sum $\sum_{k=1}^i \psi_1(x_k)$, combined with the intrinsic dynamics $g(z)$.
Overall, this model captures the coupled evolution of a structured population and an external variable, where the population evolves according to a branching mechanism and, in turn, influences the dynamics of the macroscopic quantity $\mathbf Z^n$ through a cumulative interaction term.

Then we have

\begin{theorem}
Let $z \in \mathbb R^d$. There exists a continuous  averaged process $\bar{\mathbf Z}^z$ such that the sequence of processes $(\mathbf{Z}^{n,z})_{n\geq 1}$ converges in law for the Skorokhod topology on $D([0,+\infty),\mathbb R^d)$ where $D([0,+\infty),\mathbb R^d)$ is equipped with the Skorokhod topology induced by the supremum norm where $\bar{\mathbf Z}^z$ is solution of the stochastic differential equation 
\[
\frac{\mathrm d\bar{\mathbf Z}_t^z}{\mathrm dt}=\bar F_{\mathbf I_t}^1(\bar{\mathbf Z}_t^z),\quad\mathbf Z_0^z=z
\]
with 
\begin{align*}
\bar F_i^1(z)=\int_{E_i}F_i^1(z,x)\,\mu_i(\mathrm dx).
\end{align*}
\end{theorem}
\begin{proof}
The result follows from Theorem~\ref{th:conv}. We verify its assumptions.
\medskip

\textit{Lipschitz condition.} For all $i\in I$, $x\in E_i$
\begin{align*}
\lvert F_i^1(y,x)-F_i^1(z,x)\rvert&=\lvert g(y)-g(z)\rvert\\
&\leq L_g\,\lvert y-z\rvert,\quad\forall y,z\in\mathbb R^d
\end{align*} 
so the assumption \ref{hyp:Lip} is satisfied with  $L_i=L_g$ for all $i\in I$. 

\medskip

\textit{Growth control.} Since $\psi_1$ is bounded, we have 
\begin{align*}
\sup_{x\in E_i}\lvert F_i(0,x)\rvert\leq i\,\|\psi_1\|_\infty\quad\forall i\geq1.
\end{align*} 
According to Proposition 4.1. in \cite{kagan2025averaging}, the sequence of processes $(\mathbf{I}^n)_{n\geq 1}$ converges in law for the Skorokhod topology on $D([0,+\infty),\mathbb N)$ to 
a standard continuous-time Galton-Watson tree with reproduction/branching rate $\chi(\beta)$ and offspring distribution \[L = \bigg(\frac{\chi(\beta K( \cdot, D^j))}{\chi(\beta)}\bigg)_{j \in \mathbb N}.\]
The convergence of the fast population process together with the Lipschitz property of $g$ ensures that all assumptions of Theorem~\ref{th:conv} are satisfied. The convergence of $\mathbf Z^n$ to $\bar{\mathbf Z}$ follows.
\end{proof}

\subsubsection{Multiplicative population–environment interaction} 

We also consider a scenario in which $\psi_2 : E \to \mathbb{R}_+$ represents the absorption capacity per unit area of each individual type such that $\|\psi_2\|_\infty<\infty$. Let $g : Z \to \mathbb{R}$ be Lipschitz continuous, i.e., there exists $L_g > 0$ such that
\begin{align*}
\lvert g(y)-g(z)\rvert\leq L_g\lvert y-z\rvert.
\end{align*} 
In this case, the interaction term is defined as
\begin{align*}
F_i^2\big(z,(x_1,\ldots,x_i)\big)
&=g(z)\sum_{k=1}^i\psi_2(x_k).
\end{align*}
where, for all $i\in\{1,\dots,\mathbf I_t^n\}$, the $x_i$ are i.i.d. copies of a stochastic process $X$.\\
Here, the macroscopic variable $\mathbf Z_t^n$ evolves according to a multiplicative effect: the intrinsic dynamics $g(z)$ are scaled by the total absorption capacity of the population. Each individual contributes proportionally to its own capacity, and the cumulative effect depends on the current state $z$. We check that assumptions \ref{hyp:Lip} and \ref{hyp:Fbound} hold true:
\begin{align*}
F_i^2\big(z_1,(x_1,\ldots,x_i)\big)-F_i^2\big(z_2,(x_1,\ldots,x_i)\big)&=g(z_1)\sum_{k=1}^i\psi_2(x_k)-g(z_2)\sum_{k=1}^i\psi_2(x_k)\\
&=\big(g(z_1)-g(z_2)\big)\sum_{k=1}^i\psi_2(x_k)
\end{align*}
so
\begin{align*}
\big\lvert F_i^2\big(z_1,(x_1,\ldots,x_i)\big)-F_i^2\big(z_2,(x_1,\ldots,x_i)\big)\big\rvert&\leq L_g\,\lvert z_1-z_2\rvert\,\sum_{k=1}^i\psi_2(x_k)\\
&\leq L_g\,i\,\|\psi_2\|_\infty\,\lvert z_1-z_2\rvert,
\end{align*}
and
\begin{align*}
\lvert F_i(z,0)\rvert&=i\,\lvert g(z)\,\psi_2(0)\rvert\\
&\leq i\,\big(\lvert g(z)-g(0)\rvert+\lvert g(0)\rvert\big)\,\lvert\psi_2(0)\rvert\\
&\leq i\,\psi_2(0)\big(\lvert z\rvert\,+\lvert g(0)\rvert\big).
\end{align*}
Then we have the following theorem
\begin{theorem}
Let $z \in \mathbb R^d$. There exists a continuous  averaged process $\bar{\mathbf Z}^z$ such that the sequence of processes $(\mathbf{Z}^{n,z})_{n\geq 1}$ converges in law for the Skorokhod topology on $D([0,+\infty),\mathbb R^d)$ where $D([0,+\infty), \mathbb R^d)$ is equipped with the Skorokhod topology induced by the supremum norm where $\bar{\mathbf Z}^z$ is solution of the stochastic differential equation where $\bar{\mathbf Z}^z$ is solution of the stochastic differential equation 
\[
\frac{\mathrm d\bar{\mathbf Z}_t^z}{\mathrm dt}=\bar F_{\mathbf I_t}^2(\bar{\mathbf Z}_t^z),\quad\mathbf Z_0^z=z
\]
with 
\begin{align*}
\bar F_i^2(z)=\int_{E_i}F_i^2(z,x)\,\mu_i(\mathrm dx).
\end{align*}
\end{theorem}

\bibliographystyle{plain} % ou alpha / abbrv
\bibliography{references}

\end{document}